\documentclass{daj}

\usepackage{amsfonts,amsmath,amsthm}
\usepackage{epsfig}
\usepackage{enumitem}

\dajAUTHORdetails{%
  title = {Finitely Forcible Graphons with an Almost Arbitrary Structure}, 
  author = {Daniel Kr{\'a}l', 
        L{\'a}szl{\'o} M. Lov{\'a}sz, 
	Jonathan A. Noel, and 
	Jakub Sosnovec},
  plaintextauthor = {Daniel Kral, Laszlo M. Lovasz,	Jonathan A. Noel,
	Jakub Sosnovec},
}   

\dajEDITORdetails{%
   year={2020},
   number={9},
   received={25 September 2019},   
   published={31 July 2020},  
   doi={10.19086/da.12058},       
}   

\newtheorem{conjecture}{Conjecture}
\newtheorem{thm}{Theorem}
\newtheorem{lem}[thm]{Lemma}
\newtheorem{prop}[thm]{Proposition}

\newcommand{\eps}{\varepsilon}
\newcommand{\FF}{\mathcal{F}}
\newcommand{\JE}{\mathcal{E}}
\newcommand{\JJ}{\mathcal{J}}
\newcommand{\CC}{\mathcal{C}}
\newcommand{\GG}{\mathcal{G}}
\newcommand{\XX}{\mathcal{X}}
\newcommand{\PP}{\mathcal{P}}
\newcommand{\RR}{\mathbb{R}}
\newcommand{\NN}{\mathbb{N}}

\newcommand{\wg}{\widetilde{g}}
\newcommand\dd{\,\mbox{d}}

\DeclareMathOperator{\pdeg}{pre-deg}
\newcommand*\diff{\mathop{}\!\mathrm{d}}

\begin{document}

\begin{frontmatter}[classification=text]

\title{Finitely Forcible Graphons with an Almost Arbitrary Structure\thanks{The work of the first and third authors has received funding from the European Research Council (ERC) under the European Union's Horizon 2020 research and innovation programme (grant agreement No 648509). This publication reflects only its authors' view; the European Research Council Executive Agency is not responsible for any use that may be made of the information it contains. The first author was also supported by the Engineering and Physical Sciences Research Council Standard Grant number EP/M025365/1. The work of the second author was supported by NSF Postdoctoral Fellowship Award DMS-1705204. The work of the fourth author was supported by the Leverhulme Trust 2014 Philip Leverhulme Prize of the first author and partially by the project 17-09142S of GA\v{C}R.}} 

\author[dan]{Daniel Kr{\'a}l'
}
\author[laci]{L{\'a}szl{\'o} M. Lov{\'a}sz
}
\author[jon]{Jonathan A. Noel
}
\author[kuba]{Jakub Sosnovec
}

\begin{abstract}
Graphons are analytic objects representing convergent sequences of large graphs.
A graphon is said to be finitely forcible if it is determined by finitely many subgraph densities, i.e., if the asymptotic structure of graphs represented by such a graphon depends only on finitely many density constraints.
Such graphons appear in various scenarios, particularly in extremal combinatorics.

Lov\'asz and Szegedy conjectured that all finitely forcible graphons possess a simple structure.
This was disproved in a strong sense by Cooper, Kr\'al' and Martins,
who showed that any graphon is a subgraphon of a finitely forcible graphon.
We strengthen this result by showing for every $\eps>0$ that
any graphon spans a $1-\eps$ proportion of a finitely forcible graphon.
\end{abstract}
\end{frontmatter}


\section{Introduction}

The theory of graph limits is an emerging area of combinatorics, which offers analytic tools to study large graphs.
The range of applications of analytic methods offered by the theory of graph limits has been constantly expanding.
The most prominent examples of such applications come from the closely related flag algebra method of Razborov~\cite{bib-razborov07},
which changed the landscape of extremal graph combinatorics by providing progress on numerous important problems in the area,
e.g.~\cite{bib-flag1, bib-flagrecent, bib-flag2, bib-flag3,bib-flag13, bib-flag4, bib-flag5, bib-flag6,
bib-flag7, bib-flag8, bib-flag9, bib-flag10, bib-razborov07,bib-flag11, bib-flag12}.
Among other applications of the methods provided by the theory,
we would like to highlight those from computer science 
related to property and parameter testing algorithms~\cite{bib-lovasz10+}.
We refer the reader to the recent monograph by Lov\'asz~\cite{bib-lovasz-book} for further results.

In this paper, we are interested in limits of sequences of dense graphs.
An analytic object representing a sequence of dense graphs is called a graphon.
Formally, a \emph{graphon} is a measurable function $W$ from the unit square $[0,1]^2$
to the unit interval $[0,1]$ that is symmetric; i.e., $W(x,y)=W(y,x)$ for every $(x,y)\in[0,1]^2$.
Given a graphon $W$, we can define the density $d(H,W)$ of a graph $H$ in $W$ (we give the definition in Section~\ref{sec-prelim}).
Every graphon is uniquely determined, up to weak isomorphism, by the densities of all graphs.
The main objects of our study are \emph{finitely forcible} graphons,
which are graphons that are uniquely determined by the densities of finitely many graphs.
We refer the reader to Section~\ref{sec-prelim} for a detailed presentation of these concepts.

Results on finitely forcible graphons can be found in disguise in various settings in graph theory.
For example, a classical result of Thomason~\cite{bib-thomason}, also see Chung, Graham and Wilson~\cite{bib-chung89+}, on quasirandom graphs
is equivalent to saying that the constant graphon is finitely forcible by the densities of $4$-vertex graphs.
Another source of motivation for studying finitely forcible graphons comes from extremal graph theory. For example,
Proposition~\ref{prop-ffext}, given in Section~\ref{sec-prelim}, 
states that a graphon $W$ is finitely forcible if and only if
there exists some linear combination of subgraph densities such that $W$ is its unique minimizer.

Lov{\'a}sz and Szegedy~\cite{bib-lovasz11+} initiated a systematic study of properties of finitely forcible graphons and
conjectured, based on examples of finitely forcible graphons known at that time, that all finitely forcible graphons must posses a simple structure.
In particular, a graphon $W$ can be associated with a topological space
whose points correspond to types of vertices that appear in any sequence of graphs converging to $W$ (so-called typical vertices);
the following conjectures assert that this space cannot be complex for finitely forcible graphons.

\begin{conjecture}[{Lov\'asz and Szegedy, \cite[Conjecture~9]{bib-lovasz11+}}]
\label{conj-compact}
The space of typical vertices of every finitely forcible graphon is compact.
\end{conjecture}
\begin{conjecture}[{Lov\'asz and Szegedy, \cite[Conjecture~10]{bib-lovasz11+}}]
\label{conj-dimension}
The space of typical vertices of every finitely forcible graphon has finite dimension.
\end{conjecture}

Conjectures~\ref{conj-compact} and~\ref{conj-dimension} were disproved by counterexample constructions in~\cite{bib-comp} and~\cite{bib-inf}, respectively.
A stronger counterexample to Conjecture \ref{conj-dimension} was given in~\cite{bib-reg}: if true, Conjecture~\ref{conj-dimension} would imply that
the minimum number of parts of a weak $\eps$-regular partition of a finitely forcible graphon is bounded by a power of $\eps^{-1}$.
For the finitely forcible graphon constructed in~\cite{bib-reg}, any weak $\eps$-regular partition must have a number of parts almost
exponential in $\eps^{-2}$ for infinitely many $\eps>0$, which is close to the general lower bound from~\cite{bib-conlon12+}.
This line of research culminated with the following general result of Cooper, Martins and the first author~\cite{ckm}.
\begin{thm}
\label{thm-ckm}
For every graphon $W_F$,
there exists a finitely forcible graphon $W_0$ such that $W_F$ is a subgraphon of $W_0$ induced by a $1/14$ fraction of the vertices of $W_0$.
\end{thm}
Theorem~\ref{thm-ckm} yields counterexamples to Conjectures \ref{conj-compact} and \ref{conj-dimension} and
provides a universal framework for constructing finitely forcible graphons with very complex structure. 
In view of Proposition~\ref{prop-ffext}, Theorem~\ref{thm-ckm} says that
problems on minimizing a linear combination of subgraph densities,
which are among the problems of the simplest kind in extremal graph theory,
may have unique optimal solutions with highly complex structure.
Furthermore, given the general nature of Theorem~\ref{thm-ckm},
it is surprising~\cite{ckm} that
the family of graphs whose densities force $W_0$ in Theorem~\ref{thm-ckm} can be chosen to be independent of $W_F$.

It is natural to ask whether the fraction $1/14$ in Theorem~\ref{thm-ckm} can be replaced by a larger quantity.
The proof techniques from~\cite{ckm} allows replacing the fraction by any number smaller than $1/2$.
The purpose of this paper is to show that the fraction can be replaced by any number smaller than $1$.
\begin{thm}
\label{thm1}
For every $\eps>0$ and every graphon $W_F$,
there exists a finitely forcible graphon $W_0$ such that $W_F$ is a subgraphon of $W_0$ induced by a $1-\eps$ fraction of the vertices of $W_0$.
\end{thm}
\noindent The proof of Theorem~\ref{thm1} is based on the method of decorated constraints, which was introduced in~\cite{bib-comp,bib-inf}, and
uses Theorem~\ref{thm-ckm} as one of the main tools.
Informally speaking, Theorem~\ref{thm-ckm} is used to embed the graphon $W_F$ on a small part of $W_0$ and
other auxiliary structure of $W_0$ is then used to magnify the graphon $W_F$ to the $1-\eps$ fraction of the vertices of $W_0$.
We remark that, in contrast to the proof of Theorem~\ref{thm-ckm}, the family of graphs used to force $W_0$ in Theorem~\ref{thm1} depends on $\eps$, and we show in Section~\ref{sect-unifamily} that this dependence is necessary.

\section{Preliminaries}
\label{sec-prelim}

We now introduce the notation and terminology used in the paper.
We start with some general notation.
For $k\in\NN$, $[k]$ denotes the set of integers $\{1,2,\ldots,k\}$.
If $\FF$ is a family of sets, we use $\bigcup\FF$ to denote the union of all sets $F\in\FF$.
Unless stated otherwise, we work with the Lebesgue measure on $[0,1]^d$ throughout the paper.
If $X\subseteq\RR^d$ is a measurable set,
we write $|X|$ for its measure and for two measurable sets $X,Y\subset\RR^d$, and
we write $X\sqsubseteq Y$ to mean $|X\setminus Y|=0$.

\subsection{Graphs and graphons}

The \emph{order} of a graph $G$, which is denoted by $|G|$, is its number of vertices.
The \emph{density} of a graph $H$ in $G$, which is denoted $d(H,G)$, is the probability that a uniformly randomly chosen set of $|H|$ vertices of $G$
induces a graph isomorphic to $H$. If $|H|>|G|$, then we set $d(H,G)$ to zero

Our notation mostly follows that used in~\cite{ckm} in relation to graph limits.
As defined earlier, a \emph{graphon} is a measurable function $W$ from the unit square $[0,1]^2$ to the unit interval $[0,1]$ that
is \emph{symmetric}; i.e., $W(x,y)=W(y,x)$ for every $(x,y)\in[0,1]^2$.
Conceptually, a graphon $W$ can be thought of as an infinite weighted graph on the vertex set $[0,1]$
with the edge $(x,y)\in[0,1]^2$ having weight $W(x,y)$.
Following this intuition, we refer to the points of $[0,1]$ as \emph{vertices}.
To visualize the structure of a graphon,
we shall use a figure that may be seen as a continuous version of the adjacency matrix.
More precisely, in a figure depicting $W$\!,
the domain of $W$ is represented by the unit square $[0,1]^2$ with the origin in the top left corner.
The values of $W$ are represented by appropriate shades of gray ($0$ corresponds to white and $1$ to black).
See the left side of Figure~\ref{decor_example} for an example.

A graphon can be associated with a probability distribution on graphs of a fixed order.
Formally, for a graphon $W$ and an integer $k\in\NN$,
a $W$-\emph{random graph} of order $k$ is a graph $G$ obtained by the following two step procedure.
Let $x_1,\ldots,x_k\in[0,1]$ be $k$ points chosen uniformly and independently at random.
Form a graph $G$ with the vertex set $[k]$ such that
the edge $ij$ is present with probability $W(x_i,x_j)$ for every pair of distinct vertices $i,j\in[k]$.
The \emph{density} of a graph $H$ in the graphon $W$ is the probability that a $W$-random graph of order $|H|$
is isomorphic to $H$; we denote the density of $H$ in $W$ by $d(H,W)$.
Note that $d(H,W)$ is also the expected density of $H$ in a $W$-random graph $G$ of order $k$ for every $k\geq |H|$.

Consider a sequence of graphs $\left(G_n\right)_{n\in\NN}$ such that the orders $\left|G_n\right|$ tend to infinity.
We say that the sequence $\left(G_n\right)_{n\in\NN}$ is \emph{convergent} if for every graph $H$,
the sequence of the densities of $H$ in $G_n$, i.e., the sequence $\left(d\left(H,G_n\right)\right)_{n\in\NN}$, is convergent.
We say that the sequence $\left(G_n\right)_{n\in\NN}$ converges to a graphon $W$ if
\[\lim_{n\to\infty}d\left(H,G_n\right)=d(H,W)\]
for every graph $H$.
Lov{\'a}sz and Szegedy~\cite{bib-lovasz06+} showed that every convergent sequence of graphs converges to a graphon.
Conversely, they showed that for every graphon, there exists a sequence of graphs converging to it.
In particular, the sequence of $W$-random graphs of increasing orders converges to $W$ with probability one.

For a graphon $W$ and a vertex $x\in[0,1]$, we define the \emph{degree} of $x$ in $W$ as
\[\deg_W(x)=\int_{[0,1]}W(x,y)\diff y.\]
Note that the degree is well-defined for almost every vertex $x\in[0,1]$.
We also define the \emph{neighbourhood} of a vertex $x\in[0,1]$ as the set $\{y\in[0,1]:W(x,y)>0\}$ and denote it $N_W(x)$.
Note that the set $N_W(x)$ is measurable for almost every $x\in[0,1]$. The \emph{density} of a graphon $W$ between two measurable subsets $A$ and $B$ of $[0,1]$ is defined to be
\[d_W(A,B)=\int_{A\times B}W(x,y)\diff{x}\diff{y}.\]
When the graphon $W$ is clear from context, we omit the subscripts.
Finally, we define a graphon parameter $t(C_4,W)$ as follows:
\[t(C_4,W)=\int_{[0,1]^4}W(x_1,x_2)W(x_2,x_3)W(x_3,x_4)W(x_4,x_1)\diff{x_1}\diff{x_2}\diff{x_3}\diff{x_4},\]
which is equal to the probability that a randomly chosen cyclic $4$-tuple in a $W$-random graph forms a (not necessarily induced) cycle of length four.
Observe that $t(C_4,W)=d(C_4,W)/3+d(K_4^-,W)/3+d(K_4,W)$, where $K_4^-$ is the graph obtained from $K_4$ by removing an edge.

We say that two graphons $W_1$ and $W_2$ are \emph{weakly isomorphic} if $d\left(H,W_1\right)=d\left(H,W_2\right)$ for every graph $H$.
Clearly, weakly isomorphic graphons are limits of the same sequences of graphs.
It is natural to ask how weakly isomorphic graphons can differ in their structure;
this was answered in~\cite{bib-borgs10+}.
Recall that a function $\varphi\colon[0,1]\to[0,1]$ is \emph{measure-preserving}
if it is measurable and $|\varphi^{-1}(X)|=|X|$ for every measurable subset $X\subseteq[0,1]$.
It is easy to check that if $\varphi\colon[0,1]\to[0,1]$ is a measure-preserving map,
then the graphon $W^\varphi$ defined as $W^\varphi(x,y)=W(\varphi(x),\varphi(y))$ is weakly isomorphic to $W$. In \cite{bib-borgs10+}, it was shown in particular that
two graphons $W_1$ and $W_2$ are weakly isomorphic if and only if
there exist measure-preserving maps $\varphi_1,\varphi_2\colon[0,1]\to[0,1]$ such that $W_1^{\varphi_1}= W_2^{\varphi_2}$ almost everywhere.

Let $W_1$ and $W_2$ be two graphons and $X\subseteq[0,1]$ a non-null measurable set.
We say that $W_1$ is a \emph{subgraphon} of $W_2$ induced by $X$
if there exist measure-preserving maps $\varphi_1\colon X\to [0,|X|)$ and $\varphi_2\colon X\to X$ such that
\[W_1\left(|X|^{-1}\cdot\varphi_1(x),|X|^{-1}\cdot\varphi_1(y)\right)=W_2\left(\varphi_2(x),\varphi_2(y)\right)\]
for almost every $(x,y)\in X\times X$. 

A graphon $W$ is \emph{finitely forcible} if there exist graphs $H_1,\ldots,H_m$ such that
any graphon $W'$ satisfying that $d\left(H_i,W'\right)=d\left(H_i,W\right)$ for every $i\in[m]$ is weakly isomorphic to $W$. 
A family of graphs $H_1,\ldots,H_m$ whose densities determine the graphon $W$ up to weak isomorphism is called a \emph{forcing family}.
Examples of finitely forcible graphons include constant graphons, step graphons~\cite{bib-lovasz08+} and the half-graphon~\cite{bib-diaconis09+,bib-lovasz11+}.
We remind the reader that
a \emph{step graphon} is a graphon $W$ such that
there exists a partition of $[0,1]$ into intervals $U_1,\ldots,U_k$
such that $W$ is constant on $U_i\times U_j$ for every $i,j\in[k]$, and
the \emph{half-graphon} is the graphon $W_\Delta$ such that $W_\Delta(x,y)=1$ if $x+y\geq 1$ and $W_\Delta(x,y)=0$ otherwise.

The following proposition provides a link between finitely forcible graphons and extremal graph theory.

\begin{prop}
\label{prop-ffext}
A graphon $W$ is finitely forcible if and only if
there exist graphs $H_1,\ldots,H_m$ and reals $\alpha_1,\ldots,\alpha_m$ such that for every graphon $W'$\!,
\[\sum_{i=1}^m \alpha_i d\left(H_i,W\right)\leq \sum_{i=1}^m \alpha_i d\left(H_i,W'\right),\]
and equality holds if and only if $W$ and $W'$ are weakly isomorphic.
\end{prop}

In the opposite direction, it is not the case that every linear combination of subgraph densities has a unique minimizer 
(e.g. $d(K_3,W)$ is minimized by any triangle-free graphon $W$). 
However, Lov\'asz~\cite{bib-lovasz-open,bib-lovasz-large,bib-lovasz-book,bib-lovasz11+} conjectured that 
one can always add further density constraints to make the solution unique. 
Specifically, he conjectured the following.
Let $H_1,\ldots,H_\ell$ be graphs and $d_1,\ldots,d_\ell$ reals.
If there exists a graphon $W$ such that $d(H_i,W)=d_i$, $i=1,\ldots,\ell$,
then there exists a finitely forcible such graphon.
This statement turned out to be false and
the conjecture was recently disproved by Grzesik and two of the authors~\cite{bib-finforceaddconstraints}.

We next give two statements about graphons that are needed in our arguments.
We start with a lemma,
which is a special case of~\cite[Proof of Theorem 3.12]{bib-dolezal+} given by Dole\v{z}al et al.
We include a short proof of the lemma for the completeness.
We remark that the lemma can also be proven following the lines of~\cite[Section 5.4]{ckm}.
We also remark that it is necessary to require that $t(C_4,W_1)=t(C_4,W_2)$ as can be seen by the following example.
Let $W_1$ be the graphon that is equal to zero on $[0,1/2]^2\cup [1/2,1]^2$ and equal to one elsewhere, and
let $W_2$ be the graphon equal to $1/2$ everywhere (see Figure~\ref{fig_opex}).
Further, let $\varphi(x)=2x\mod 1$.
It is easy to check that $d_{W_1}(\varphi^{-1}(J),\varphi^{-1}(J'))=d_{W_2}(J,J')=\frac{|J|\cdot|J'|}{2}$
for all measurable subsets $J,J'\subseteq [0,1]$, however,
the graphons $W_1$ and $W_2$ are not weakly isomorphic.

\begin{figure}[ht]
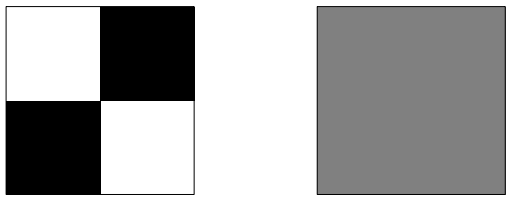

\begin{center}
\epsfbox{epsuniversal-66.mps} \hskip 15mm
\epsfbox{epsuniversal-67.mps}
\end{center}
\caption{The graphons $W_1$ and $W_2$ from the example given before Lemma~\ref{lm-oper}.}
\label{fig_opex}
\end{figure}

\begin{lem}
\label{lm-oper}
Let $W_1$ and $W_2$ be two graphons.
Suppose that there exists a measure preserving map $\varphi:[0,1]\to [0,1]$ such that
\begin{equation} \label{eq-operlemcond}
d_{W_1}(\varphi^{-1}(J),\varphi^{-1}(J'))=d_{W_2}(J,J')
\end{equation}
for all measurable subsets $J,J'\subseteq [0,1]$.
If $t(C_4,W_1)=t(C_4,W_2)$, then for almost all $(x,y)\in [0,1]^2$, it holds that $W_1(x,y)=W_2(\varphi(x),\varphi(y))$.
\end{lem}

\begin{proof}
A graphon $W$ can be viewed as an operator $T_W$ on $L^2([0,1])$, the $L^2$-space of functions from $[0,1]$,
defined as 
\[T_W(h)(x)=\int_{[0,1]} W(x,y)h(y)\diff y\]
for $h\in L^2([0,1])$.
The operator $T_W$ is self-adjoint and compact for any graphon $W$~\cite[Section 7.5]{bib-lovasz-book}.
Moreover, it has a discrete spectrum, all its eigenvalues $(\lambda_i)_{i\in\NN}$ are real and
it holds~\cite[Section 7.5]{bib-lovasz-book} that
\begin{equation} \label{eq-operlemC4count}
t(C_4,W)=\sum_{i=1}^{\infty}\lambda_i^4.
\end{equation} 

The map $\varphi :[0,1] \to [0,1]$ 
naturally yields a pullback embedding
$\varphi^*:L^2([0,1]) \to L^2([0,1])$,
where $\varphi^*(h)(x):=h(\varphi(x))$ for $h\in L^2([0,1])$.
Similarly, we define $\varphi^*(W)(x,y):=W(\varphi(x),\varphi(y))$ for a graphon $W$.
Let $K$ be the image of $L^2([0,1])$ under $\varphi^*$;
$K$ is a subspace of $L^2([0,1])$, and a standard argument shows that it is closed.
Equation \eqref{eq-operlemcond} implies that for any $h_1,h_2\in K$,
we have $\langle T_{\varphi^*(W_2)} h_1,h_2\rangle = \langle T_{W_1}h_1,h_2 \rangle $. Furthermore, if $h$ is orthogonal to $K$, then $T_{\varphi^*(W_2)}h=0$. Let $\Pi_K$ be the orthogonal projection onto $K$.
We then have that 
\[T_{\varphi^*(W_2)}=\Pi_K T_{W_1} \Pi_K. \]
This implies that if 
$(\lambda_i)_{i\in\NN}$ are the eigenvalues of $T_{W_1}$ and
$(\mu_i)_{i\in\NN}$ the eigenvalues of $T_{\varphi^*(W_2)}$, listed in non-increasing order according to their absolute values,
then $|\mu_i|\le|\lambda_i|$ for every $i\in\NN$, and
equality holds for every $i$ if and only if $T_{\varphi^*(W_2)}=T_{W_1}$. In particular,
\[\sum_{i\in\NN}\mu_i^4\le\sum_{i\in\NN}\lambda_i^4,\]
and equality holds if and only if
$T_{\varphi^*(W_2)}=T_{W_1}$.
By equation \eqref{eq-operlemC4count},
we do indeed have equality, so $T_{\varphi^*(W_2)}=T_{W_1}$, and therefore $\varphi^*(W_2)=W_1$ almost everywhere.
\end{proof}

We next state a property of weakly isomorphic graphons that
we will need in our arguments presented further in the paper.
For a graphon $W$, let $\omega(W)$ be the supremum of $|A|$ taken over all measurable sets $A\subseteq [0,1]$ such that
$W$ is equal to $1$ almost everywhere on $A\times A$.

\begin{lem}
\label{lm-omega}
If $W$ and $W'$ are two weakly isomorphic graphons,
then $\omega(W)=\omega(W')$.
\end{lem}

\begin{proof}
By the results from \cite{bib-borgs10+} on weakly isomorphic graphons,
it suffices to show that
if $W$ is a graphon and $\varphi:[0,1] \to [0,1]$ is a measure-preserving map, then $\omega(W)=\omega(W^\varphi)$.
To show this, we use that 
\[\omega(W)=\sup\left\{\|f\|_1\bigg|f\in L^1([0,1]),0\le f\le 1,\int_{[0,1]^2} f(x)\left(1-W(x,y)\right)f(y)\dd x\dd y=0\right\}.\]
Every function $f$ satisfying that
 \[\int_{[0,1]^2} f(x)\left(1-W(x,y)\right)f(y)\dd x\dd y=0\] 
also satisfies that
\[\int_{[0,1]^2} f(\varphi(x))\left(1-W^\varphi(x,y)\right)f(\varphi(y))\dd x\dd y=0.\]
Since it holds that $\|f(\varphi(\cdot))\|_1=\|f\|_1$, we obtain that $\omega(W^\varphi)\ge\omega(W)$.
Conversely,
if $g:[0,1] \to [0,1]$ is a measurable function,
there exists a measurable function $f:[0,1] \to [0,1]$ such that
\[\int_A f(x)\dd x=\int_{\varphi^{-1}(A)}g(x)\dd x\]
for any measurable set $A$.
Indeed, the function $f$ is the Radon-Nikodym derivative of the pushforward measure $\varphi_*(\nu)$ according to the Lebesgue measure,
where the measure $\nu$ is defined as $\nu(X)=\int_X g$.
Note that $\|g\|_1=\|f\|_1$.
Furthermore, if a function $g$ satisfies
\[\int_{[0,1]^2} g(x)\left(1-W^\varphi(x,y)\right)g(y)\dd x\dd y=0,\]
then the corresponding function $f$ satisfies that
\[\int_{[0,1]^2} f(x)\left(1-W(x,y)\right)f(y)\dd x\dd y=0.\]
This implies $\omega(W) \ge \omega(W^\varphi)$, and
we can conclude that $\omega(W)=\omega(W^\varphi)$ as desired.
\end{proof}

We close this subsection with two propositions that we need in our exposition.
\begin{prop}
\label{densall}
Let $H$ be a graph on $n$ vertices.
There exist connected graphs $H_1,\ldots,H_k$ with at most $n$ vertices each and a polynomial $p$ in $k$ variables such that
the following holds for every graphon $W$:
\[d(H,W)=p\left(d(H_1,W),\ldots,d(H_k,W)\right).\]
\end{prop}
\begin{proof}
We proceed by induction on the number of connected components of $H$.
The base of induction is the case that $H$ is connected,
in which case we set $k=1$, $H_1=H$ and $p$ to be the identity.
Suppose that $H$ is not connected;
let $H'$ be one of its components and let $H''$ be the subgraph of $H$ induced by the vertices not in $H'$.
Further, let $n'$ be the number of vertices of $H'$.
Observe that
\[d(H',W)\cdot d(H'',W)=
  \sum_{G}\sum_{\substack{A\subseteq V(G), |A|=n'\\
                G[A]\approx H'\\
		G[V(G)\setminus A]\approx H''}}\binom{n}{n'}^{-1}d(G,W),\]
where the sum is taken over all graphs $G$ on $n$ vertices,
$G[A]\approx H'$ represents that the subgraph of $G$ induced by $A$ is isomorphic to $H'$, and
$G[V(G)\setminus A]\approx H''$ represents that the subgraph of $G$ induced by $V(G)\setminus A$ is isomorphic to $H''$.
Note that an $n$-vertex graph $G$ has a subset $A$ of $n'$ vertices such that $G[A]\approx H'$ and $G[V(G)\setminus A]\approx H''$
only if $G$ is $H$ or $G$ has fewer components than $H$.
It follows that $d(H,W)$ can be expressed as a linear combination of 
$d(G,W)$, where $G$ ranges through $n$-vertex graphs with fewer components than $H$, and
$d(H',W)\cdot d(H'',W)$;
the coefficients of this linear combination do not depend on $W$.
The proposition now follows.
\end{proof}
The second proposition is a well-known measure-theoretic result which we will apply throughout the paper.
It follows from the Monotone Reordering Theorem~\cite[Proposition A.19]{bib-lovasz-book} and
the fact that any standard probability space is isomorphic to the unit interval.
\begin{prop}
\label{prop-MRT}
Let $X$ be a non-null measurable subset of $[0,1)$ and $h:X\to\RR$ a measurable function.
There exists a measure-preserving map $\varphi:X\to [0,|X|)$ and
a non-decreasing function $f:[0,|X|)\to\RR$ such that
$h(x)=f(\varphi(x))$ for almost every $x\in X$.
\end{prop}

\subsection{Partitioned graphons and decorated constraints}

The most direct way of showing that a graphon $W$ is finitely forcible
is by explicitly providing the forcing family of graphs $H_1,\ldots,H_m$ and their densities $d_1,\ldots,d_m$ and
analyzing all graphons $W'$ such that $d\left(H_i,W'\right)=d_i$.
However, this approach often becomes impractical when $m$ is very large and, even more so, when
$H_1,\dots,H_m$ depend on $\varepsilon$, as is required to prove Theorem~\ref{thm1}.
We now introduce the method of decorated constraints that was developed in~\cite{bib-inf, bib-comp},
which allows us to use more advanced constraints to establish that a graphon is finitely forcible.

A \emph{density expression} is a formal polynomial in graphs; i.e., 
graphs and real numbers  are density expressions, and
if $D_1$ and $D_2$ are density expressions, then so are $D_1+D_2$ and $D_1\cdot D_2$.
A density expression $D$ can be \emph{evaluated} with respect to a graphon $W$ by replacing every graph $H$ in $D$ with $d(H,W)$.
A \emph{constraint} is an equality between two density expressions.
A constraint is \emph{satisfied} by a graphon $W$ if both density expressions evaluated with respect to $W$ are equal.
A simple example of a constraint is the equality $H=d$, which is satisfied by a graphon $W$ if and only if $d(H,W)=d$.

If $\mathcal{C}$ is a finite set of constraints and $W$ is the unique graphon, up to weak isomorphism, that
satisfies all of the constraints in $\mathcal{C}$, then $W$ is finitely forcible.
Indeed, $W$ is the unique graphon, up to weak isomorphism, with the density of $H$ equal to $d(H,W)$ for all graphs $H$ appearing in a constraint in $\mathcal{C}$.
We will often say that the constraints contained in $\mathcal{C}$ \emph{force} the graphon $W$.

A graphon $W$ is \emph{partitioned}
if there exist positive reals $a_1,\ldots,a_k$ summing to one and distinct reals $d_1,\ldots,d_k\in[0,1]$ such that
the set $A_i$ of vertices of $W$ with degree $d_i$ has measure~$a_i$ for all $i\in [k]$.
The sets $A_i$ are called \emph{parts} and the \emph{degree} of a part $A_i$ is $d_i$.
We will abuse the notation here and if $W$ and $W'$ are two partitioned graphons with parts of measures $a_i$ and degrees $d_i$,
we will use the same letters to denote the corresponding parts of $W$ and $W'$.
This is technically incorrect since the part $A_i$ can be a different subset of $[0,1]$ in $W$ and $W'$
but we will make sure that the graphon that we have in mind is always clear from the context. The property of being a partitioned graphon can be forced in the following sense; see~\cite[Lemma 2]{bib-comp} for a proof.

\begin{lem}\label{lemma_partitioned}
Let $a_1,\ldots,a_k$ be positive reals summing to one and $d_1,\ldots,d_k$ distinct reals from $[0,1]$.
There exists a finite set of constraints $\mathcal{C}$ such that
a graphon $W$ satisfies all constraints in $\mathcal{C}$ if and only if
$W$ is a partitioned graphon with parts of measures $a_1,\ldots,a_k$ and degrees $d_1,\ldots,d_k$.
\end{lem}

Consider a partitioned graphon $W$ and let $\PP$ be the set of its parts.
The \emph{relative degree} of a vertex $x\in[0,1]$ in $W$ with respect to a non-empty set $\XX\subseteq\PP$ of parts is defined as
\[\deg_W^\XX(x)=\left|\bigcup\XX\right|^{-1}\cdot\int_{\bigcup\XX}W(x,y)\diff y.\]
Similarly, the \emph{relative neighbourhood} of $x\in[0,1]$ with respect to $\XX$, which is denoted by $N_W^\XX(x)$, is the set $N_W(x)\cap\bigcup\XX$.
If $\XX=\{X\}$ for some part $A$, then we simply write $\deg_W^X(x)$ and $N_W^X(x)$.
As before, if the graphon $W$ is clear from context, then we omit the subscripts.
For two non-empty subsets $\XX_1,\XX_2\subseteq\PP$,
the restriction of the graphon $W$ to $\bigcup\XX_1\times\bigcup\XX_2$ will be referred to as the \emph{tile} $\XX_1\times\XX_2$.
If both $\XX_1$ and $\XX_2$ are singletons, we call the tile \emph{simple}; otherwise, it is \emph{composite}. 

We now introduce a formally stronger (but technically equivalent) version of constraints, which we call decorated constraints.
These are similar to decorated constraints used in~\cite{bib-reg,ckm,bib-inf, bib-comp} except that we will allow 
vertices of graphs appearing in constraints to be assigned to multiple parts, 
as opposed to just a single part.
We discuss the difference in more detail further.
We will always have a particular set $\PP$ of parts in mind when working with decorated constraints.
A \emph{decorated} graph $G$ is a graph with $0\leq m\leq |G|$ distinguished vertices labelled from $1$ to $m$,
which are called \emph{roots}, and
with every vertex $v$ (including the roots) assigned a non-empty subset of $\PP$,
which is called the \emph{decoration} of $v$.
If the decoration of a vertex is a single element set, e.g., $\{A\}$, we just write $A$ as the decoration to simplify our notation.
Two decorated graphs $G_1$ and $G_2$ are \emph{compatible}
if the subgraphs induced by their roots are isomorphic, respecting both the labels of roots and the decorations assigned to them.
A \emph{decorated density expression} is a formal polynomial in decorated graphs such that
all graphs in the expression are mutually compatible, and
a \emph{decorated constraint} is an equality between two decorated density expressions such that all graphs in the expression are mutually compatible.

Let $W$ be a partitioned graphon with parts $\PP$ and $C$ a decorated constraint of the form $D=0$ where $D$ is a decorated density expression.
We now describe what we mean when we say that the graphon $W$ satisfies $C$.
Let $H_0$ be the decorated graph induced by the roots $v_1,\ldots,v_m$ of the decorated graphs in $C$. Call an $m$-tuple $\left(x_1,\ldots,x_m\right)\in[0,1]^m$ \emph{feasible} if each $x_i$ belongs to one of the parts that $v_i$ is decorated with,
$W\left(x_i,x_j\right)>0$ for every edge $v_iv_j\in E\left(H_0\right)$ and $W\left(x_i,x_j\right)<1$ for every non-edge $v_iv_j\notin E\left(H_0\right)$. Given a feasible $m$-tuple $\left(x_1,\ldots,x_m\right)\in[0,1]^m$, the evaluation of $D$ at the $m$-tuple is obtained by replacing each decorated graph $H$ with the probability that a $W$-random graph of order $|H|$ is the graph $H$, conditioned on the event that
the roots are chosen as the vertices $x_1,\ldots,x_m$, that they induce the graph $H_0$, and that
each non-root vertex is chosen from the union of the parts in its decoration.
The graphon $W$ \emph{satisfies} the constraint $C$ if
for almost every feasible $m$-tuple,
the evaluation of $D$ is equal to zero.
We say that $W$ satisfies a decorated constraint of the form $D=D'$ if it satisfies $D-D'=0$.

We next describe a convention of depicting decorated constraints that we use in this paper,
which is analogous to that used in~\cite{bib-reg,ckm}. The roots of decorated graphs will be represented by squares and the non-root vertices by circles. The decoration of every vertex will be depicted as a label inside the square or circle. If $X$ is any letter such that $X_1, X_2, \ldots$ are parts, then $X_*$ will denote the label $\{X_1,X_2,\ldots\}$. For example, if $B_1, B_2$ are all the parts with the letter $B$, then $B_*$ will refer to the label $\{B_1,B_2\}$. The roots in all decorated graphs appearing in a constraint will be placed on the same mutual positions; i.e., the corresponding roots of different graphs in the constraint are on the same respective positions. Edges are represented as solid lines between vertices and non-edges are represented as dashed lines. The absence of any line between two root vertices indicates that the constraint should hold for both cases when the edge between the root vertices is present and when it is not present. Finally, the absence of a line between a non-root vertex and another vertex represent the sum of decorated graphs with this edge present and without this edge. Thus, if $k$ such lines are absent in a decorated graph, the figure represents the sum of $2^k$ decorated graphs.

We now give an example of evaluating decorated constraints, which is illustrated in Figure~\ref{decor_example}.
We consider the graphon $W$ depicted in the left part of the figure:
the graphon $W$ has three parts $A$, $B_1$ and $B_2$, each of measure $1/3$.
The densities between the parts are as given in the figure.
In particular, the degree of $A$ is $2/3$, the degree of $B_1$ is $11/18$ and the degree of $B_2$ is $1$.
We next consider the decorated graph $H$ depicted in the right part of Figure~\ref{decor_example}.
The graph $H$ has two roots $v_1$ and $v_2$ that are adjacent and decorated with $B_1$ and $A$, respectively, and
it has two non-root vertices $v_3$ and $v_4$ that are also adjacent and decorated with $\{B_1,B_2\}$ (denoted $B_*$) and $A$, respectively.
The vertex $v_3$ is adjacent to both roots and $v_4$ is adjacent only to $v_1$.
The probability described in the previous paragraph is independent of the choice of $x_1$ and $x_2$ in $B_1$ and $A$ and is equal to $13/96$.
In particular, the graphon $W$ satisfies the decorated constraint $H=13/96$ depicted in Figure~\ref{decor_example}.

\begin{figure}
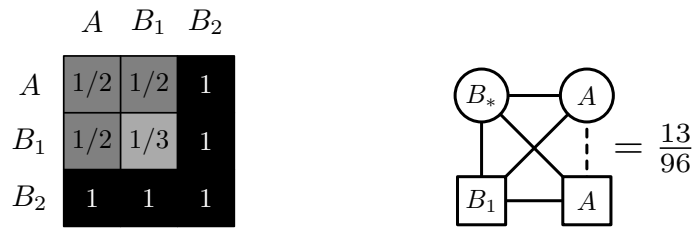

\begin{center}
\hskip -60mm \epsfxsize 3cm \epsfbox{epsuniversal-47.mps} 
\vskip -22mm \hskip 60mm  \epsfbox{epsuniversal-48.mps}		
\end{center}
\caption{An example of evaluating a decorated constraint.}
\label{decor_example}
\end{figure}

In~\cite[Lemma 3]{bib-comp}, it was shown that decorated constraints where each vertex is decorated with a single element set are equivalent to (ordinary) constraints.
Let us call such decorated constraints \emph{simple}; i.e.,
a decorated constraint is simple if all decorations appearing in it are single element sets.
We now argue that each decorated constraint is equivalent to a set of simple decorated constraints.
If a decorated graph $H$ contains a non-root vertex $v$ decorated by a set of parts $\XX$,
we may replace $H$ with a convex combination of graphs $H$ decorated by elements of $\XX$,
where the coefficients are proportional to the measures of the parts from $\XX$.
If one of the roots, say $v$, appearing in a decorated constraint is decorated by a set $\XX$,
we consider all decorated constraints with $v$ labelled by elements of $\XX$.
In this way, we can convert any decorated constraint to an equivalent set of simple decorated constraints.
Hence, we can conclude, using~\cite[Lemma 3]{bib-comp}, that the following holds.

\begin{prop}\label{lem_set_decorated}
Fix the number of parts and their sizes and degrees.
For every decorated constraint $C$,
there exists a finite collection of constraints $\CC'$ such that
a partitioned graphon $W$ satisfies $C$ if and only if it satisfies $\CC'$.
\end{prop}

We next describe how decorated constraints can be used to embed a finitely forcible graphon inside another graphon.
Suppose that $W_0$ is a finitely forcible graphon that is forced by constraints $H_i=d_i$ for $i\in[k]$, where $H_1,\ldots,H_k$ are graphs and $d_1,\ldots,d_k$ are their densities.
If $W$ is a partitioned graphon and $A$ one of its parts,
replacing each $H_i$ with the decorated graph where each vertex is decorated with $A$ results in a set of constraints that
are satisfied if and only if the subgraphon of $W$ induced by $A$ is weakly isomorphic to $W_0$.
The same holds if instead of a single part $A$ we consider a set of parts.
We state this observation as a separate lemma.

\begin{lem}\label{lemma_part}
Let $W_0$ be a finitely forcible graphon, $\PP$ a set of parts and $\XX$ a non-empty subset of $\PP$.
There exists a finite set $\CC$ of decorated constraints such that
every partitioned graphon $W$ with parts $\PP$ satisfies $\CC$ if and only if
the subgraphon of $W$ induced by $\bigcup \XX$ is weakly isomorphic to $W_0$.
\end{lem}

We conclude this subsection by stating a lemma,
which appeared implicitly in~\cite[proof of Lemma 2.7 or Lemma 3.3]{bib-lovasz11+} and was explicitly stated in~\cite[Lemma 8]{bib-reg}.

\begin{lem}\label{lemma_int}
Let $X,Z\subseteq\mathbb{R}$ be measurable non-null sets and let $F\colon X\times Z\to [0,1]$ be a measurable function.
If there exists a constant $C\in\mathbb{R}$ such that
\[\int_Z F(x,z)F(x',z)\diff{z}=C\]
for almost every $(x,x')\in X^2$, then it holds that
\[\int_Z F(x,z)^2\diff{z}=C\]
for almost every $x\in X$.
\end{lem}

\section{Main Proof}

In this section, we prove Theorem~\ref{thm1}.
For technical reasons, it is easier to consider graphons as functions from $[0,1)^2$ rather than $[0,1]^2$, and
we do so throughout the section.
Note that this change affects a graphon on a set of measure zero only.
In general, we refer to intervals of the type $[a,b)$ as \emph{half-open},
however, we do not refer to intervals of the type $(a,b]$ as \emph{half-open}.

For the proof of Theorem~\ref{thm1}, 
fix a graphon $W_F$ and $\eps>0$.
We can assume that $\frac{1}{\eps}-1$ is a power of two, i.e., $\eps=\frac{1}{2^{r}+1}$ for an integer $r$, and that
almost every vertex of $W_F$ has degree less than one.
If either assumption does not hold, choose $\eps'<\eps$ such that $\frac{1}{\eps'}-1$ is a power of two and
apply the theorem with the graphon $W'_F$ such that
$W'_F(x,y)=W_F(\frac{1-\eps'}{1-\eps}\cdot x,\frac{1-\eps'}{1-\eps}\cdot y)$
for $(x,y)\in [0,\frac{1-\eps}{1-\eps'})^2$ and $W'_F(x,y)=0$ elsewhere.

Set $M=4\left(\frac{1}{\eps}-1\right)$ and $m=\log_2 M$.
By applying the Monotone Reordering Theorem and, if needed, changing the graphon $W_F$ on a set of measure zero,
we can assume that there exists a partition of $[0,1)$ into half open intervals $Q_1,\ldots,Q_M$ such that
the degree of every vertex $x$ of $W_F$ contained in $Q_k$, $k\in [M]$, belongs to the interval $[(k-1)/M,k/M)$ and
the subinterval $Q_k$ precedes $Q_{k+1},\ldots,Q_M$ for every $k\in [M]$.
Note that some of the subintervals $Q_k$ can be empty.

\subsection{Overview of \texorpdfstring{$\boldsymbol{W_0}$}{W0}}
\label{subsection_general}

We next provide a description of the general structure of the graphon $W_0$ and
present the detailed definition of individual tiles throughout this section
together with the decorated constraints enforcing its structure.
We also refer the reader to Figure~\ref{fig1}, where the graphon $W_0$ is visualized, and
to Table~\ref{tab-ref}, which provides references to subsections where individual tiles are forced.
The graphon $W_0$ is a partitioned graphon with 
\begin{itemize}
\item $M$ parts $A_1,\ldots,A_M$,
\item $M+9$ parts $B_A, \ldots, B_F, B_{G_1},\ldots, B_{G_M}, B_P, B_Q, B_R$,
\item $m+1$ parts $C_1,\ldots,C_m$ and $C_\infty$,
\item $m+1$ parts $D_1,\ldots,D_m$ and $D_\infty$,
\item $m$ parts $E_1,\ldots,E_{m-1}$ and $E_\infty$,
\item $M$ parts $F_1,\ldots,F_M$, and
\item two parts $G_1$ and $G_2$.
\end{itemize}
In total, $W_0$ has $3M+3m+13$ parts, and
the set of the parts contained in each of the seven groups above is denoted $A_*, B_*, \ldots, G_*$, respectively;
this notation is in line with the notation that we have introduced for visualizing decorated constraints.
The set of all $3M+3m+13$ parts of $W_0$ is denoted by $\PP$, i.e.,
$\PP=A_*\cup B_*\cup\cdots\cup G_*$.
We will also use $B_{G_*}$ to denote the set containing the parts $B_{G_1},\ldots, B_{G_M}$, and
$|X_*|$ to denote the measure of the union of the parts contained in $X_*$ for $X\in\{A,B,\ldots,G\}$.
For some arguments that we present,
it may be convenient to think of parts contained in each of the groups as a single part.
Indeed, the parts contained in the same group serve a similar purpose.

\begin{figure}
\begin{center}
\epsfxsize 14cm
\epsfbox{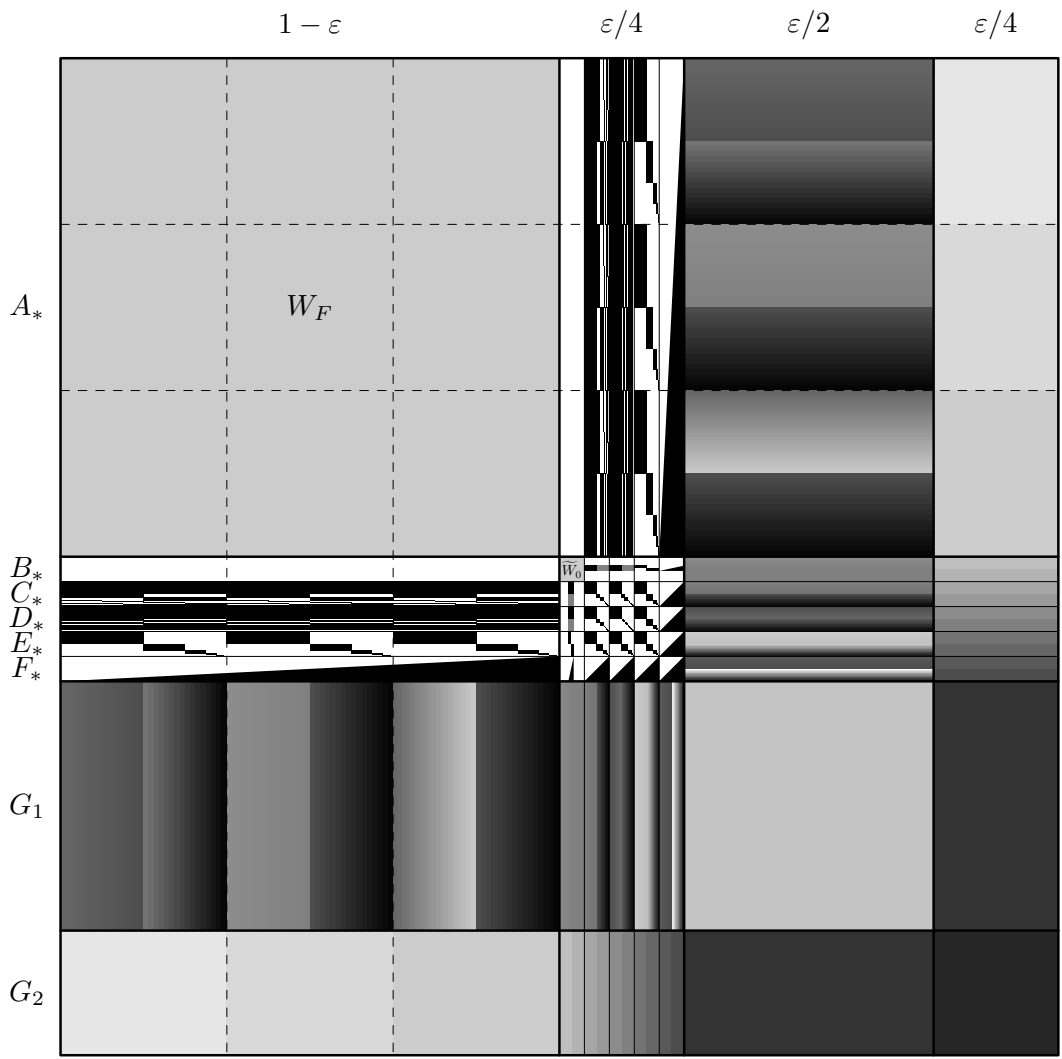}
\end{center}
\caption{A sketch of the graphon $W_0$.
         The composite tile $A_*\times A_*$ contains the graphon $W_F$ and
	 the composite tile $ B_*\times B_*$ contains the graphon ${\widetilde W}_0$ from Theorem~\ref{CKM}.
	 An enhanced sketch of the structure between the parts contained in $A_*$ and $B_{G*}$ and
	 the parts contained in $C_*$, $D_*$, $E_*$ and $F_*$ is given in Figure~\ref{fig2}.
	 \label{fig1}}
\end{figure}

\begin{figure}
\begin{center}
\epsfbox{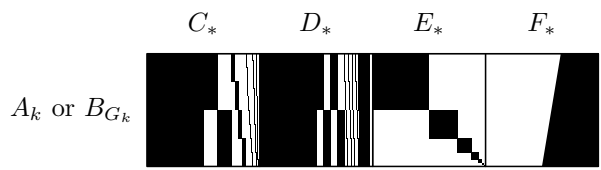}
\end{center}
\caption{A sketch of the structure of the graphon $W_0$
         between the parts contained in $A_*$ and $B_{G*}$ and
         the parts contained in $C_*$, $D_*$, $E_*$ and $F_*$.\label{fig2}}
\end{figure}

\begin{table}
\begin{center}
\begin{tabular}{|c|cccccccc|}
\hline
      & $A_*$ & $B_*$ & $C_*$ & $D_*$ & $E_*$ & $F_*$ & $G_1$ & $G_2$ \\
\hline
$A_*$ & \ref{forcing_densities} & \ref{section_balancing} & \ref{section_squares} & \ref{section_squares} & \ref{section_checker} & \ref{coordinate_system} & \ref{section_balancing} & \ref{section_distinguishing} \\
$B_*$ & \ref{section_balancing} & \ref{ff_section} & \ref{section_squares} & \ref{section_squares} & \ref{section_checker} & \ref{coordinate_system} & \ref{section_balancing} & \ref{section_distinguishing} \\
$C_*$ & \ref{section_squares} & \ref{section_squares} & \ref{section_indices} & \ref{section_indices} & \ref{section_checker} & \ref{coordinate_system} & \ref{section_balancing} & \ref{section_distinguishing} \\
$D_*$ & \ref{section_squares} & \ref{section_squares} & \ref{section_indices} & \ref{section_indices} & \ref{section_checker} & \ref{coordinate_system} & \ref{section_balancing} & \ref{section_distinguishing} \\
$E_*$ & \ref{section_checker} & \ref{section_checker} & \ref{section_checker} & \ref{section_checker} & \ref{section_checker} & \ref{coordinate_system} & \ref{section_balancing} & \ref{section_distinguishing} \\
$F_*$ & \ref{coordinate_system} & \ref{coordinate_system} & \ref{coordinate_system} & \ref{coordinate_system} & \ref{coordinate_system} & \ref{coordinate_system} & \ref{section_balancing} & \ref{section_distinguishing} \\
$G_1$ & \ref{section_balancing} & \ref{section_balancing} & \ref{section_balancing} & \ref{section_balancing} & \ref{section_balancing} & \ref{section_balancing} & \ref{section_balancing} & \ref{section_distinguishing} \\
$G_2$ & \ref{section_distinguishing} & \ref{section_distinguishing} & \ref{section_distinguishing} & \ref{section_distinguishing} & \ref{section_distinguishing} & \ref{section_distinguishing} & \ref{section_distinguishing} & \ref{section_distinguishing} \\
\hline
\end{tabular}
\end{center}
\caption{References to the subsections where the corresponding tiles are analyzed.}
\label{tab-ref}
\end{table}

We now describe the structure of the graphon $W_0$.
Each part $X\in\PP$ of the graphon is a half-open subinterval of $[0,1)$ with measure given in Table~\ref{table2}, and
these subintervals follow the order in which they were listed when we introduced the parts of $W_0$.
For the rest of this section, the subinterval of $[0,1)$ forming a part $X\in\PP$ of $W_0$ is denoted by $X^0$;
we use $X^0$ solely to denote the subinterval while we still use $X$ for the part of $W_0$.
This allows us to clearly distinguish the subintervals of $[0,1)$ forming the parts of $W_0$
from the subsets of $[0,1)$ corresponding to the parts with the same name in other graphons that we will consider.
We will also use $A_*^0,\ldots,G_*^0$ to denote the unions of the subintervals associated
with the parts contained in $A_*,\ldots,G_*$, respectively.

\begin{table}
\centering
\begin{tabular}{|c|c|c|}
\hline
Part & Measure & Pre-Degree\\
\hline
$A_k$ & $(1-\eps) |Q_k|$ & $\frac{\eps(k+1)}{4}$ \\
$B_{G_k}$ & $\frac{\eps}{20}\cdot\frac{1}{14}\cdot |Q_k|$ & $\frac{\eps}{4}$ \\
$B_A,\ldots,B_F,B_P,B_R$ & $\frac{\eps}{20}\cdot\frac{1}{14}$ & $\frac{\eps}{4}$ \\
$B_Q$ & $\frac{\eps}{20}\cdot\frac{5}{14}$ & $\frac{\eps}{4}$ \\
$C_k, D_k$ & $\frac{\eps}{20}\cdot\frac{1}{2^k}$ & $\frac{1-\eps}{2^{k-1}}+\frac{\eps}{4}$ \\
$C_\infty, D_\infty$ & $\frac{\eps}{20}\cdot\frac{1}{2^m}$ & $\frac{1-\eps}{2^m}+\frac{\eps}{4}=\frac{\eps}{2}$ \\
$E_k$ & $\frac{\eps}{20}\cdot \frac{1}{2^k}$ & $\frac{1-\eps}{2^k}+\frac{\eps}{4}$ \\
$E_\infty$ & $\frac{\eps}{20}\cdot\frac{1}{2^{m-1}}$ & $\frac{1-\eps}{2^m}+\frac{\eps}{4}=\frac{\eps}{2}$ \\
$F_k$ & $\frac{\eps}{20}\cdot \frac{1}{M}$ & $\frac{\eps(k+1)}{4}$ \\
$G_1$ & $\frac{\eps}{2}$ & \\
$G_2$ & $\frac{\eps}{4}$ & \\
\hline
\end{tabular}
\caption{The sizes and the pre-degrees of the parts of the graphon $W_0$.}
\label{table2}
\end{table}

The graphon $W_F$ is contained on the composite tile $A_*\times A_*$ and
we set $W_0((1-\eps)x,(1-\eps)y)=W_F(x,y)$ for every $(x,y)\in [0,1)^2$.
In this way,
the part $Q_k$ of the graphon $W_F$ corresponds to the part $A_k$ of the graphon $W_0$ for every $k\in [M]$.
The graphon $W_0$ outside the composite tile $A_*\times A_*$ will be defined in the following subsections, and
we use the convention that when the value $W_0(x,y)$ is defined, the definition also sets the value $W_0(y,x)$.

The parts contained in $B_*\cup\cdots\cup F_*$ of the graphon $W_0$ are used to enforce its structure, and
the parts $G_1$ and $G_2$ are used to balance the degrees inside the parts.
For each part $X\in\PP$ except for $G_1$ and $G_2$,
we define a real number $\pdeg(X)$, which we call the \emph{pre-degree} of $X$.
These numbers are given in Table~\ref{table2}.
The definition of the graphon $W_0$ will ensure that
\begin{equation}
\label{eq_predeg}
\int_{[0,1)\setminus G_2^0}W_0(x,z)\diff{z}=\pdeg(X)
\end{equation}
for every $x\in X^0$; further details are given in Subsection~\ref{section_balancing}.
We next fix an irrational number $\delta_X\in (0,\eps/4)$ for each part $X\in\PP$ such that
the numbers $\delta_X$, $X\in\PP$, are rationally independent;
in particular, all the numbers $\delta_X$, $X\in\PP$, are mutually distinct.
The part $G_2$ will be used to distinguish different parts of the graphons
by guaranteeing that the degree of each part $X\in\PP\setminus\{G_1,G_2\}$ is $\pdeg(X)+\delta_X$.
The graphon $W_0$ is constant on each tile $X\times G_2$, $X\in\PP$, and
the sole purpose of these tiles is to guarantee that different parts have distinct degrees.

We conclude this subsection by defining a notation that will be convenient in our exposition.
Let $\XX$ be a non-empty set of parts of $W_0$ and $\XX^0$ the set of corresponding half-intervals.
If $\bigcup\XX^0$ is a half-open interval, we define a mapping $\gamma_{\XX}:[0,1)\to\bigcup\XX^0$ as
\begin{equation}
\label{eq-inttopart}
\gamma_{\XX}(x)=x\cdot\left|\bigcup\XX^0\right|+\min\bigcup\XX^0\;.
\end{equation}
Informally speaking, $\gamma_{\XX}$ maps the half-interval $[0,1)$ to the half-interval $\bigcup\XX^0$ linearly.
For example, $W_0(\gamma_{A_*}(x),\gamma_{A_*}(y))=W_F(x,y)$ for every $x,y\in[0,1)$.
If $\XX=\{X\}$, we will just write $\gamma_X$ instead of $\gamma_{\{X\}}$.

\subsection{Universal graphon}
\label{CKM_section}

In this subsection, we revisit the construction of the graphon $W_0$ from Theorem~\ref{thm-ckm} given in~\cite{ckm}.
In the proof of Theorem~\ref{thm1},
we apply Theorem~\ref{thm-ckm} with the same graphon $W_F$ for which we are proving Theorem~\ref{thm1}.
To distinguish the graphons $W_0$ from Theorems~\ref{thm-ckm} and~\ref{thm1},
we will be using ${\widetilde W}_0$ for the graphon from Theorem~\ref{thm-ckm}.
The graphon ${\widetilde W}_0$ obtained in this way is visualized in Figure~\ref{fig-ckm}, and
we now review some of the properties of the graphon ${\widetilde W}_0$ and
the proof of its finite forcibility given in~\cite{ckm}.

\begin{figure}
\begin{center}
\epsfbox{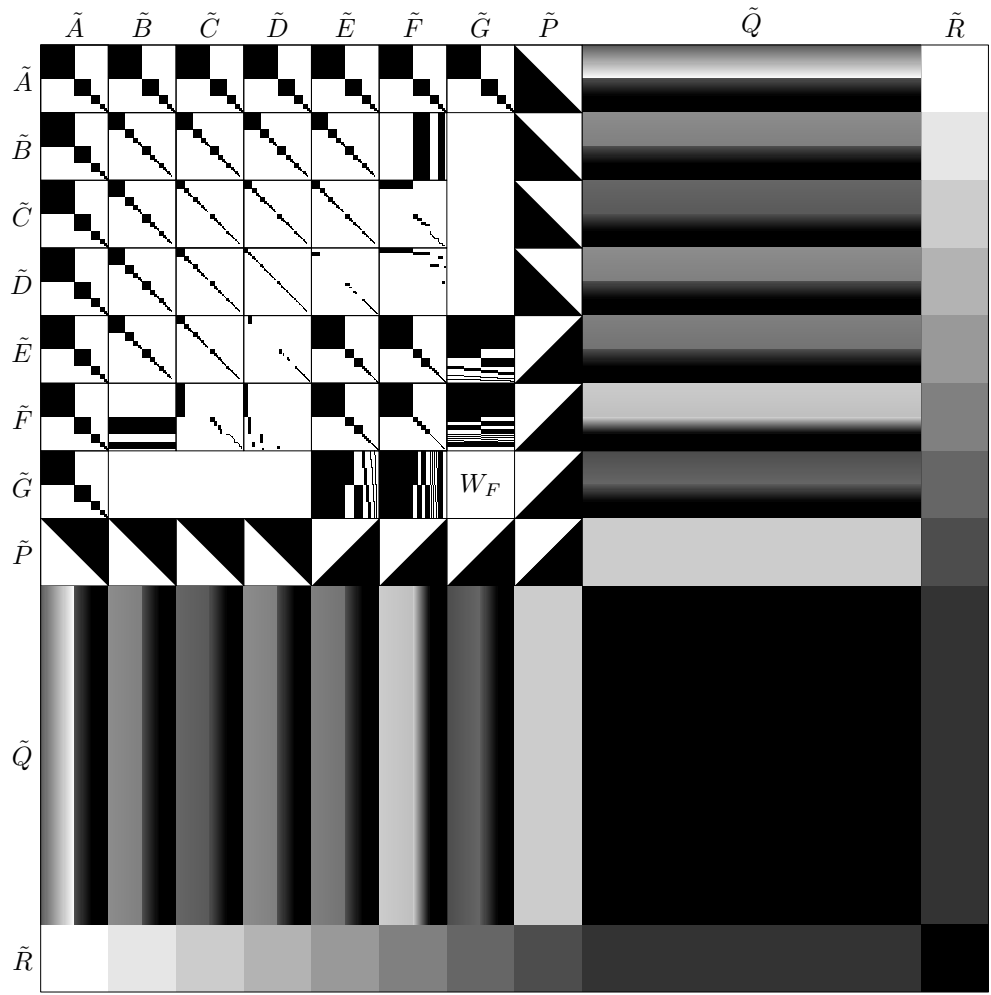}
\end{center}
\caption{The graphon ${\widetilde W}_0$ constructed in~\cite{ckm}.}
\label{fig-ckm}
\end{figure}

\begin{thm}
\label{CKM}
The graphon ${\widetilde W}_0$ is a partitioned graphon
with $10$ parts ${\widetilde A},\ldots,{\widetilde G}$, ${\widetilde P}$, ${\widetilde Q}$ and ${\widetilde R}$ that
has the following properties in particular.
\begin{enumerate}[label=(\alph*)]
\item
  The parts ${\widetilde A},\ldots,{\widetilde G}$, ${\widetilde P}$ and ${\widetilde R}$ are
  half-open intervals $[0/14,1/14)$, $\ldots$, $[6/14,7/14)$, $[7/14,8/14)$ and $[13/14,14/14)$, respectively.
  In particular, each of these parts has measure $1/14$.
\item The part ${\widetilde Q}$ is $[8/14,13/14)$, i.e., its measure is $5/14$.
\item
  It holds that 
  \[{\widetilde W}_0\left(\frac{6+x}{14},\frac{6+y}{14}\right)=W_F(x,y)\]
  for every $(x,y)\in [0,1)^2$, i.e., the subgraphon induced by ${\widetilde G}$ is $W_F$.
\item
  It holds that 
  \[{\widetilde W}_0\left(\frac{6+x}{14},\frac{7+y}{14}\right)=W_{\Delta}(x,y)\]
  for every $(x,y)\in [0,1)^2$, where $W_{\Delta}$ is the half-graphon defined in Section~\ref{sec-prelim},
  i.e., the tile ${\widetilde G}\times {\widetilde P}$ is the half-graphon.
\item\label{CKM_item}
  For every graphon $W$ that is weakly isomorphic to ${\widetilde W}_0$,
  there exists a measure-preserving map ${\wg}:[0,1)\to [0,1)$ such that
  \[W(x,y)={\widetilde W}_0\left({\wg}(x),{\wg}(y)\right)\]
  for almost every $(x,y)\in[0,1)^2$.
\end{enumerate}
\end{thm}

We use the graphon ${\widetilde W}_0$ to define $W_0$ on the composite tile $B_*\times B_*$
by setting
\[W_0(\gamma_{B_*}(x),\gamma_{B_*}(y))={\widetilde W}_0(x,y)\]
for every $[0,1)^2$.
In this way,
the parts $\widetilde A,\ldots,\widetilde F$, $\widetilde P$, $\widetilde Q$ and $\widetilde R$ of the graphon ${\widetilde W}_0$
correspond to the parts $B_A,\ldots,B_F$, $B_P$, $B_Q$ and $B_R$ of the graphon $W_0$, respectively, and
the part $\widetilde G$ to the union $B_{G_1}\cup\cdots\cup B_{G_M}$.

\subsection{General structure of \texorpdfstring{$\boldsymbol{W_0}$}{W0}}
\label{ff_section}

In this subsection, we provide an overview of the constraints that witness the finite forcibility of $W_0$, and
use some of them to establish the general structure of any graphon satisfying them.
The constraints that we use are the following:
\begin{itemize}
\item the constraints given in Lemma~\ref{lemma_partitioned} such that any graphon satisfying them
      is partitioned graphon with parts $\PP$ that have the same degrees and measures as those of $W_0$,
\item the decorated constraints from Lemma~\ref{lemma_part} applied to the graphon ${\widetilde W}_0$ and with $\XX=B_*$, and
\item the decorated constraints that we present in the current and subsequent subsections of this section.
\end{itemize}
Suppose that $W$ is a graphon satisfying all these constraints.
We will construct a particular measure-preserving map $g:[0,1)\to [0,1)$ and
prove that $W(x,y)=W_0(g(x),g(y))$ for almost every $(x,y)\in [0,1)^2$.

We now present the construction of the map $g$.
By Lemma~\ref{lemma_partitioned}, the graphon $W$ is a partitioned graphon with parts $\PP$ that
have the same measures and degrees as those in $W_0$.
In the rest of the section, the subset of $[0,1)$ forming a part $X\in\PP$ of $W$ is denoted by $X$;
recall that the half-interval forming the corresponding part in $W_0$ is denoted by $X^0$.
The Monotone Reordering Theorem implies that for each part $X\in\PP$
there exists a measure-preserving map $\varphi_X:X\to [0,|X^0|)$ and
a non-decreasing function $f_X:[0,|X^0|)\to\RR$ such that
\[f_X(\varphi_X(x))=\deg_{W}^{F_*}(x)\]
for almost every $x\in X$.
Theorem~\ref{CKM} implies that there exists a measure-preserving map ${\wg}:B_*\to [0,|B_*|)$ such that
\[W(x,y)=\widetilde W_0\left(\frac{\wg(x)}{|B_*|},\frac{\wg(y)}{|B_*|}\right)
\]
for almost every $(x,y)\in B_*\times B_*$.
We next define the mapping $g:[0,1)\to [0,1)$ as follows.
\[g(x)=\left\{
       \begin{array}{cl}
       \gamma_{B_*}({\wg}(x)/|B_*|) & \mbox{if $x$ belongs to a part contained in $B_*$, and} \\
       \gamma_X(\varphi_X(x)/|X|) & \mbox{if $x$ belongs to a part $X\not\in B_*$}.
       \end{array}
       \right.\]
Recall that our goal is to show that $W(x,y)=W_0(g(x),g(y))$ for almost every $(x,y)\in [0,1)^2$.
Note that the definition of $g$ directly implies that $W(x,y)=W_0(g(x),g(y))$ for almost every $(x,y)\in B_*\times B_*$.

\begin{figure}
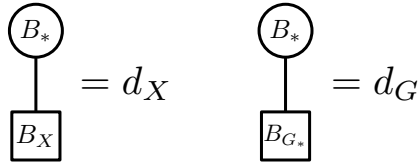

\begin{center}
\epsfbox{epsuniversal-35.mps} \hskip 10mm
\epsfbox{epsuniversal-36.mps} \hskip 10mm
\end{center}
\caption{Set-decorated constraints aligning the parts contained in $B_*$.
         The value of $X$ ranges among $A,\ldots,F,P,Q,R$, and
	 the value of $d_Z$ is equal to the degree of the part $\widetilde Z$ in the graphon ${\widetilde W}_0$
	 where $Z\in\{A,\ldots,G,P,Q,R\}$.
         \label{CKM_constraints1}}
\end{figure}

The definition of the function $g$ implies that
each part $X\not\in B_*$ of $W$ is mapped to the part $X^0$ of $W_0$ by $g$,
however, this property is not implied by the definition of $g$ for the parts $X\in B_*$.
The decorated constraints presented in Figure~\ref{CKM_constraints1} guarantee this as we will now argue.
Recall that for almost every $x,y \in B_* \times B_*$,
we have $W(x,y)=\widetilde W_0\left(\frac{\wg(x)}{|B_*|},\frac{\wg(y)}{|B_*|}\right)$.
The first constraint in the figure implies that for every $X\in\{A,\ldots,F,P,Q,R\}$,
each vertex of the part $B_X$ of the graphon $W$
belongs to the part ${\widetilde X}$ of the graphon ${\widetilde W}_0$,
which is embedded in the composite tile $B_*\times B_*$.
Likewise, the second constraint implies that each vertex of one of the parts of $B_{G*}$
belongs to the part ${\widetilde G}$ of the graphon ${\widetilde W}_0$.
It follows that for each $X\in\{A,\ldots,F,P,Q,R\}$,
the part $B_X$ of $W$ is mapped by $g$ to the part $B_X^0$ of $W_0$, and that $B_{G_*}$ is mapped to $B_{G_*}^0$. Note that we have not yet proven that each $B_{G_i}$ is mapped to $B_{G_i}^0$; we will do so in the next section.

\subsection{Coordinate system}
\label{coordinate_system}

In this subsection,
we introduce some structure of the graphon $W_0$ that allows us to define a coordinate system inside most of its parts.
The arguments follow lines similar to those in~\cite{bib-reg,ckm,bib-comp,bib-inf}.
By Lemma~\ref{lemma_part},
there exist decorated constraints such that the composite tile $F_*\times F_*$
is weakly isomorphic to the half-graphon.
These constraints guarantee that the subgraphon induced by $F_*$ is the half-graphon,
however, they do not fix the order of the parts $F_1,\ldots,F_M$ inside it.
So, in addition to the constraints given by Lemma~\ref{lemma_part},
we also include the constraints depicted in Figure~\ref{fig_coordinate1}.

\begin{figure}
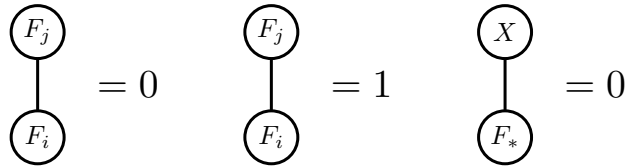

\begin{center}
\epsfbox{epsuniversal-53.mps} \hskip 10mm
\epsfbox{epsuniversal-54.mps} \hskip 10mm
\epsfbox{epsuniversal-56.mps} 
\end{center}
\caption{Decorated constraints forcing the structure of some of the tiles involving parts from $F_*$.
         The first constraint should hold for all $i,j\in [M]$ such that $i+j\le M$,
	 the second for all $i,j\in [M]$ such that $i+j\ge M+2$, and
	 the last for all $X\in\{B_A,\ldots,B_F,B_P,B_Q,B_R\}$.
         \label{fig_coordinate1}}
\end{figure}

\begin{figure}
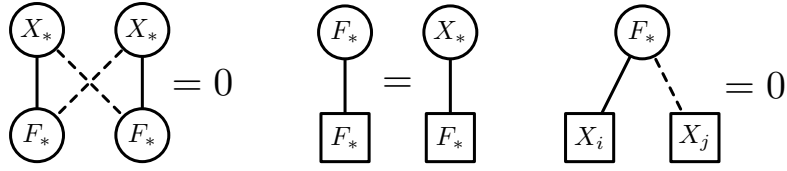

	\begin{center}
		\epsfbox{epsuniversal-46.mps} \hskip 10mm
		\epsfbox{epsuniversal-45.mps} \hskip 10mm
		\epsfbox{epsuniversal-55.mps} 
	\end{center}
\caption{Decorated constraints forcing the structure of the rest of the tiles involving parts from $F_*$.
The constraints should hold for $X\in\{A, B_G, C, D, E\}$. The last constraint should hold
for all $i<j$ with $i,j\in [M]$ if $X\in \{A,B_G\}$, with $i,j\in [m]\cup\{\infty\}$ if $X \in \{C,D\}$ and
with $i,j\in [m-1]\cup\{\infty\}$ if $X=E$ (using the convention that $i<\infty$ for every $i\in\NN$).
\label{fig_coordinate2}}
\end{figure}

The first constraint in Figure~\ref{fig_coordinate1} implies that each of the tiles $F_i\times F_j$
with $i,j\in [M]$ and $i+j\le M$ is equal to zero almost everywhere and
the second constraint implies that each of the tiles $F_i\times F_j$
with $i,j\in [M]$ and $i+j\ge M+2$ is equal to one almost everywhere.
Since the graphon on the composite tile $F_*\times F_*$ is weakly isomorphic to the half-graphon by Lemma~\ref{lemma_part},
the choice of $\varphi_{X}$ for $X\in F_*$ yields that
$W(x,y)=W_0(g(x),g(y))$ for almost every $(x,y)\in F_*\times F_*$.

The composite tiles $X\times F_*$ of $W_0$ are equal to zero for $X\in\{B_A,\ldots,B_F,B_P,B_Q,B_R\}$ and
this is enforced by the last constraint Figure~\ref{fig_coordinate1}.
Therefore, it holds that $W(x,y)=0=W_0(g(x),g(y))$ for almost every $(x,y)\in X\times F_*$
for $X\in\{B_A,\ldots,B_F,B_P,B_Q,B_R\}$.

Next fix $X\in\{A, B_G, C, D, E\}$.
For $(x,y)\in X^0_*\times F^0_*$,
we define $W_0(x,y)=1$ for $\gamma_{X_*}^{-1}(x)/|X_*|+\gamma_{F_*}^{-1}(y)/|F_*|\ge 1$, and
$W_0(x,y)=0$ otherwise.
The first constraint in Figure~\ref{fig_coordinate2} implies that
for almost every pair $y,y'\in F_*$,
either $N_{X_*}(y)\sqsubseteq N_{X_*}(y')$ or $N_{X_*}(y')\sqsubseteq N_{X_*}(y)$.
The second constraint implies that $\deg^{F_*}(y)=\deg^{X_*}(y)$ for almost every $y\in F_*$.
It follows that the composite tile $X_*\times F_*$ is a scaled half-graphon.
The last constraint guarantees that
the degrees relative to $F_*$ of almost all the vertices contained in $X_i$ are smaller than those in $X_j$ for $i<j$,
i.e., the parts of $X_*$ are ordered in the same way in $W$ as in $W_0$ according to the degrees relative to $F_*$.
In particular, this implies that $g$ maps the part $B_{G_i}$ to the part $B^0_{G_i}$ for every $i\in [M]$.
The choice of $\varphi_{X}$ for $X\in A_*\cup C_*\cup D_*\cup E_*$ according to the Monotone Reordering Theorem
yields that $W(x,y)=W_0(g(x),g(y))$ for almost every $(x,y)\in X\times F_*$.
Finally, since $g$ maps the part $B_{G_i}$ to the part $B^0_{G_i}$ for every $i\in [M]$,
we also obtain that $W(x,y)=W_0(g(x),g(y)))$ for almost every $(x,y)\in B_{G_*}\times F_*$.

\subsection{Checker tiles}
\label{section_checker}

The \emph{checker graphon} $W_C$ is the graphon defined as follows;
the graphon is also depicted in Figure~\ref{checker}.
Let $I_k$ denote the half-open interval $\left[1-2^{-k+1},1-2^{-k}\right)$ for $k\in\NN$.
Set $W_C(x,y)=0$ if $x$ and $y$ belongs to the same interval $I_k$ for some $k\in\NN$, and
$W_C(x,y)=0$ otherwise.

\begin{figure}
\begin{center}
\epsfbox{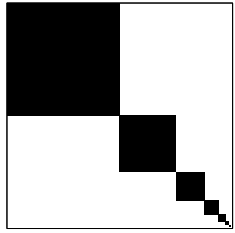}
\end{center}
\caption{The checker graphon $W_C$.\label{checker}}
\end{figure}

We now define the graphon $W_0$ on the tiles involving the parts from $E_*$.
Given a part $X$ in $\{A_1,\ldots,A_M,B_{G_1},\ldots,B_{G_M},C_*,D_*,E_*\}$,
set $W_0(\gamma_{E_*}(x),\gamma_{X}(y))=W_C(x,y)$ for all $(x,y)\in [0,1)^2$.
We also set $W(x,y)=0$ for all $(x,y)\in E_*\times X$ where $X\in B_*\setminus B_{G_*}$.

\begin{figure}[ht]
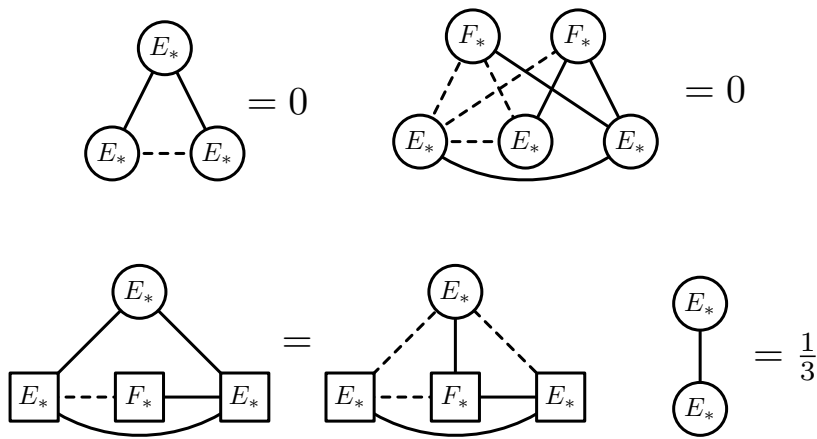

\begin{center}
\epsfbox{epsuniversal-6.mps} \hskip 10mm
\epsfbox{epsuniversal-7.mps} \vskip 10mm
\epsfbox{epsuniversal-8.mps} \hskip 10mm
\epsfbox{epsuniversal-9.mps}
\end{center}
\caption{The decorated constraints forcing the structure of the composite tile $ E_*\times  E_*$.\label{fig_checker1}}
\end{figure}

We next consider the constraints depicted in Figures~\ref{fig_checker1} and~\ref{fig_checker2}.
Since the arguments follow the lines of those presented in~\cite{ckm},
we present them here on a general level and refer the reader for further details to~\cite[Section 3.2]{ckm}.
The first constraint on the first line in Figure~\ref{fig_checker1} implies that
there exists a collection $\JE$ of disjoint measurable subsets of $E_*$ such that
the following holds for almost every $(x,y)\in E_*\times E_*$:
$W(x,y)=1$ if and only if $x$ and $y$ belong to the same set of $\JE$.
The second constraint on the first line implies that
the sets in $\JE$ are intervals with respect to the relative degrees to $F_*$ of vertices in $E_*$,
i.e., there exists a collection $\JJ$ of disjoint subintervals of $[0,1)$ such that
the following holds for almost every $(x,y)\in E_*\times E_*$:
$W(x,y)=1$ if and only if $\gamma_{E_*}^{-1}(g(x))$ and $\gamma_{E_*}^{-1}(g(y))$ belong to the same $J\in\JJ$.
The first constraint on the second line yields that the length $|J|$ of each interval $J\in\JJ$,
which is the value of the left side of the expression, is equal to $1-\sup J$,
which is the value of the right side of the expression (when the two $E_*$-roots are mapped by $\gamma_{E_*}^{-1}(g(\cdot))$ to $J$).
Finally, the remaining constraint is equivalent to saying that
\[\sum_{J\in\JJ}|J|^2=\frac{1}{3}.\]
Together with the fact that $|J|=1-\sup J$,
this implies that, up to changing each contained in $\JJ$ on a set of measure zero,
$\JJ$ contains exactly the sets $I_k$, $k\in\NN$, which were defined at the beginning of this subsection.
It follows that $W(x,y)=W_0(g(x),g(y))$ for almost every $(x,y)\in E_*\times E_*$.

\begin{figure}[ht]
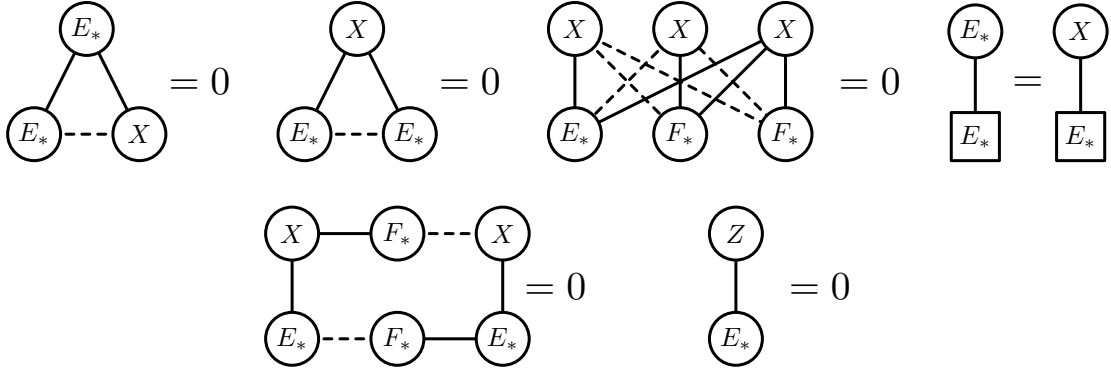

\begin{center}
\epsfbox{epsuniversal-10.mps} \hskip 5mm
\epsfbox{epsuniversal-58.mps} \hskip 5mm
\epsfbox{epsuniversal-12.mps} \hskip 5mm
\epsfbox{epsuniversal-11.mps} \vskip 5mm
\epsfbox{epsuniversal-13.mps} \hskip 15mm
\epsfbox{epsuniversal-57.mps}
\end{center}
\caption{The decorated constraints forcing the structure of
         the composite tiles $E_*\times X$ for $X\in\{A_1,\ldots,A_M,B_{G_1},\ldots,B_{G_M},C_*,D_*,E_*\}$, and
	 $E_*\times Z$ for $Z\in B_*\setminus B_{G*}$.\label{fig_checker2}}
\end{figure}

Finally, we consider the constraints depicted in Figure~\ref{fig_checker2}.
Fix $X$ to be one of $A_1,\ldots,A_M$, $B_{G_1},\ldots,B_{G_M}$, $C_*$, $D_*$ and $E_*$.
The first two constraints on the first line imply that
there exist disjoint measurable subsets $K_J\subseteq X$ such that
the following holds for almost every $(x,y)\in E_*\times X$:
$W(x,y)=1$ if and only if there exists $J\in\JJ$ such that $\gamma_{E*}^{-1}(g(x))\in J$ and $y\in K_J$.
The third constraint yields that each of these sets is an interval with respect to the relative degrees to $F_*$,
i.e., there exist disjoint subintervals $L_J\subseteq [0,1)$ such that
the following holds for almost every $(x,y)\in E_*\times X$:
$W(x,y)=1$ if and only if there exists $J\in\JJ$ such that $\gamma_{E*}^{-1}(g(x))\in J$ and $\gamma_{X}^{-1}(g(y))\in L_J$.
Furthermore, the last constraint on the first line implies that $|J|=|L_J|$ for every $J\in\JJ$.

The first constraint on the second line implies that
if $J\in\JJ$ precedes $J'\in\JJ$, then $L_J$ precedes $L_{J'}$.
Hence, we conclude that $J$ and $L_J$ differ on a set of measure zero for every $J\in\JJ$.
It follows that $W(x,y)=W_0(g(x),g(y))$ for almost every $(x,y)\in E_*\times X$.
The last constraint in Figure~\ref{fig_checker2} implies that
$W(x,y)=0=W_0(g(x),g(y))$ for almost every $(x,y)\in E_*\times Z$, $Z\in B_*\setminus B_{G_*}$.

\subsection{Exponential checker tiles}
\label{section_indices}

\begin{figure}
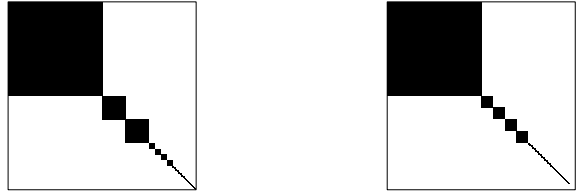

\begin{center}
\epsfbox{epsuniversal-14.mps}\hskip 25mm
\epsfbox{epsuniversal-15.mps}
\end{center}
\caption{The graphons $W_C^2$ and $W_C^3$.}
\label{fig_CD1}
\end{figure}

We next define a refined version $W_C^r$, $r\in\NN$, of the checker graphon.
Informally speaking, we form $W_C^r$ by splitting the parts of the checker graphon into $1,2^{r-1},2^{2(r-1)},2^{3(r-1)}$, etc. parts.
So, fix $r\in\NN$ and define the graphon $W_C^r$ as follows:
$W_C^r(x,y)=1$ if and only if $x$ and $y$ belongs to the same interval $I_k$, $k\in\NN$, and
\[\left\lfloor\frac{x-\min I_k}{|I_k|}\cdot 2^{(k-1)(r-1)}\right\rfloor=\left\lfloor\frac{y-\min I_k}{|I_k|}\cdot 2^{(k-1)(r-1)}\right\rfloor,\]
and $W_C^r(x,y)=0$ otherwise.
The graphons $W_C^2$ and $W_C^3$ are depicted in Figure~\ref{fig_CD1}.
Also note that the graphon $W_C^1$ is the checker graphon itself.
Finally, we define $W_0(\gamma_{C_*}(x),\gamma_{C_*}(y))=W_0(\gamma_{C_*}(x),\gamma_{D_*}(y))=W_C^2(x,y)$ and
$W_0(\gamma_{D_*}(x),\gamma_{D_*}(y))=W_C^3(x,y)$ for $(x,y)\in [0,1)^2$.

\begin{figure}
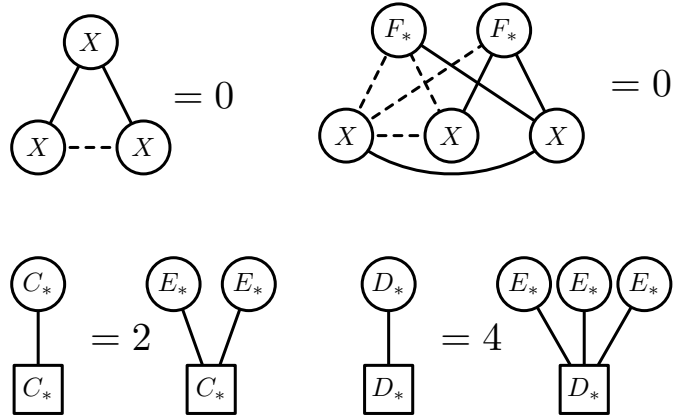

\begin{center}
\epsfbox{epsuniversal-16.mps} \hskip 10mm
\epsfbox{epsuniversal-17.mps} \vskip 10mm
\epsfbox{epsuniversal-19.mps} \hskip 10mm
\epsfbox{epsuniversal-21.mps}
\end{center}
\caption{The decorated constraints forcing the tiles $C_*\times  C_*$ and $ D_*\times D_*$.
         The constraints on the first line should hold for $X=C_*$ and $X=D_*$.}
\label{fig_CD2}
\end{figure}

Consider now the decorated constraints given in Figure~\ref{fig_CD2}.
The arguments are similar to those presented in Subsection~\ref{section_checker},
so we present them briefly.
Fix $X$ to be $C_*$ or $D_*$.
The constraints on the first line in Figure~\ref{fig_CD2} imply that
there exists a family $\JJ_X$ of disjoint subintervals $[0,1)$ such that
the following holds for almost every $(x,y)\in X\times X$:
$W(x,y)=1$ if $\gamma_X^{-1}(g(x))$ and $\gamma_X^{-1}(g(y))$ belong to the same interval $J\in\JJ_X$, and
$W(x,y)=0$ otherwise.
Without loss of generality, we can assume that all intervals in $\JJ_X$ are half-open.
The first constraint on the second line implies that
the following holds for almost every $x\in C_*$:
if $\gamma_{C_*}^{-1}(g(x))\in I_k$, $k\in\NN$, and $\gamma_{C_*}^{-1}(g(x))\in J$, $J\in\JJ_{C_*}$,
then $|J|=2|I_k|^2=2^{-2k+1}$.
It follows that $W(x,y)=W_0(g(x),g(y))$ for almost every $(x,y)\in C_*\times C_*$.
Similarly, the second constraints implies that
the following holds for almost every $x\in D_*$:
if $\gamma_{D_*}^{-1}(g(x))\in I_k$, $k\in\NN$, and $\gamma_{D_*}^{-1}(g(x))\in J$, $J\in\JJ_{D_*}$,
then $|J|=4|I_k|^3=2^{-3k+2}$.
This yields that $W(x,y)=W_0(g(x),g(y))$ for almost every $(x,y)\in D_*\times D_*$.

\begin{figure}[ht]
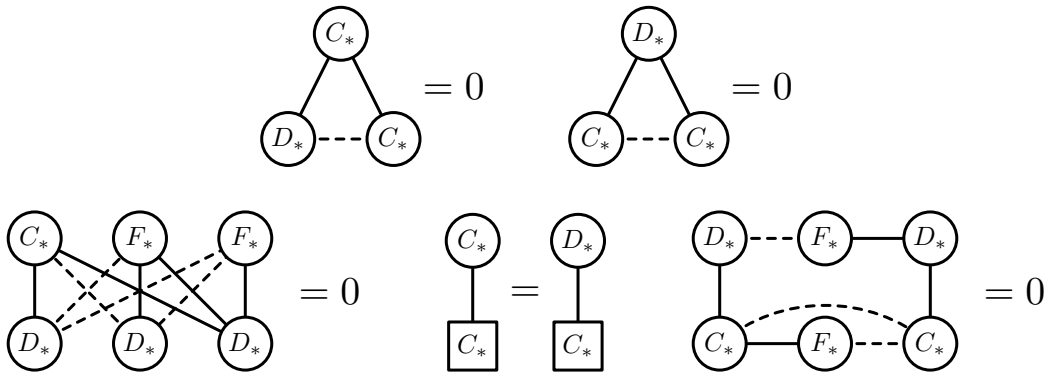

\begin{center}
\epsfbox{epsuniversal-23.mps} \hskip 10mm
\epsfbox{epsuniversal-24.mps} \vskip 5mm	
\epsfbox{epsuniversal-25.mps} \hskip 10mm
\epsfbox{epsuniversal-26.mps} \hskip 10mm
\epsfbox{epsuniversal-27.mps}
\end{center}
\caption{The decorated constraints forcing the structure of the composite tile $ C_*\times  D_*$.}
\label{fig_CD3}
\end{figure}

We next consider the constraints depicted in Figure~\ref{fig_CD3}.
The two constraints on the first line imply that there exist disjoint measurable sets $K_J\subseteq D_*$, $J\in\JJ_{C*}$, such that
the following holds for almost every $(x,y)\in C_*\times D_*$:
$W(x,y)=1$ if and only if there exists $J\in\JJ_{C_*}$ such that
$\gamma_{C_*}^{-1}(g(x))\in J$ and $y\in K_J$.
The first constraint on the second line implies that each $K_J$ is an interval with respect to the degrees relative to $F_*$,
i.e., there exist disjoint subintervals $L_J$ of $[0,1)$, $J\in\JJ_{C_*}$, such that
the following holds for almost every $(x,y)\in C_*\times D_*$:
$W(x,y)=1$ if there exists $J\in\JJ_{C_*}$ such that
$\gamma_{C_*}^{-1}(g(x))\in J$ and $\gamma_{D_*}^{-1}(g(y))\in L_J$, and $W(x,y)=0$ otherwise.
The second constraint on the second line yields that $|J|=|L_J|$.
Finally, the last constraint on the second line implies that if $J$ precedes $J'$, $J,J'\in\JJ_{C_*}$, then $L_J$ precedes $L_{J'}$.
We conclude that $W(x,y)=W_0(g(x),g(y))$ for almost every $(x,y)\in C_*\times D_*$.

\subsection{Referencing dyadic squares}
\label{section_squares}

In this subsection, we introduce an auxiliary structure that
allows us to copy a tile $B_{G_i}\times B_{G_j}$ to a tile $A_i\times A_j$, $i,j\in [M]$.
Define $I_{s,t}$ for $s\in\NN$ and $t\in [2^{s-1}]$ to be the half-open interval $\left[\frac{t-1}{2^{s-1}},\frac{t}{2^{s-1}}\right)$.
A half-open interval $I_{s,t}$ for $s\in\NN$ and $t\in [2^{s-1}]$ is a \emph{dyadic interval of order $s$}, and
$I_{s,t}\times I_{s,t'}$ for $s\in\NN$ and $t,t'\in [2^{s-1}]$ is a \emph{dyadic square of order $s$}.
Note that dyadic squares of order $s$ partition $[0,1)^2$ for every $s\in\NN$; see Figure~\ref{fig_dsq}.
A tile $B_{G_i}\times B_{G_j}$ is copied to a tile $A_i\times A_j$, $i,j\in [M]$,
by ensuring that the density of every dyadic square of the tile $A_i\times A_j$
is the same as the density of the corresponding dyadic square of the tile $B_{G_i}\times B_{G_j}$.
The auxiliary structure that is used to achieve this is embedded
in the composite tiles $X\times C_*$ and $X\times D_*$, $X\in A_*\cup B_{G_*}$.

\begin{figure}[ht]
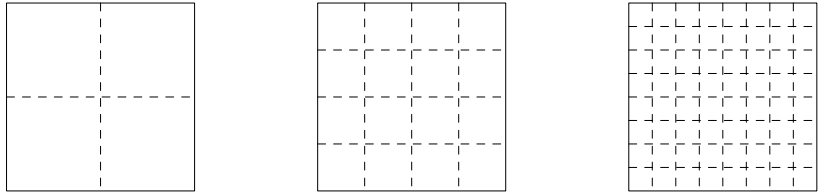

\begin{center}
\epsfbox{epsuniversal-63.mps} \hskip 15mm
\epsfbox{epsuniversal-64.mps} \hskip 15mm
\epsfbox{epsuniversal-65.mps}
\end{center}
\caption{Partitions of the square $[0,1)^2$ into dyadic squares of orders $2$, $3$ and $4$.}
\label{fig_dsq}
\end{figure}

Fix $X\in A_*\cup B_{G_*}$.
Informally speaking,
the structure between $X$ and $\gamma_{C_*}(I_s)$ and
the structure between $X$ and $\gamma_{D_*}(I_s)$
partitions $X$ into dyadic intervals of order $s$;
here, the interval $I_s$ is as defined in Subsection~\ref{section_checker}.
The structure of the tiles is illustrated in Figure~\ref{referencing_fig1}.
For every $s\in\NN$,
each dyadic interval of order $s$ appears once between $X$ and $\gamma_{C_*}(I_s)$ in the tile $X\times C_*$,
while it appears $2^{s-1}$ times between $X$ and $\gamma_{D_*}(I_s)$ in the tile $X\times D_*$.
In particular, $\gamma_{C_*}(I_s)$ is split into $2^{s-1}$ parts in the same way as in the tile $C_*\times C_*$ and
$\gamma_{D_*}(I_s)$ is split into $2^{2(s-1)}$ parts in the same way as in the tile $D_*\times D_*$.
Formally, we define the tile $X\times C_*$ as follows.
For $(x,y)\in X\times C_*$,
we define $W_0(x,y)=1$ if there exist $s\in\NN$ and $t\in[2^{s-1}]$ such that
$\gamma^{-1}_{C_*}(y)\in I_s$,
\[\frac{\gamma^{-1}_{C_*}(y)-\min I_s}{|I_s|}\in I_{s,t}\quad\mbox{and}\quad\gamma^{-1}_X(x)\in I_{s,t}\;,\]
and $W_0(x,y)=0$ otherwise.
Similarly, for $(x,y)\in X\times D_*$,
we define $W_0(x,y)=1$ if there exist $s\in\NN$ and $t\in[2^{s-1}]$ such that
$\gamma^{-1}_{D_*}(y)\in I_s$,
\[\left(\frac{\gamma^{-1}_{D_*}(y)-\min I_s}{|I_s|}\cdot 2^{s-1}\mod 1\right)\in I_{s,t}\quad\mbox{and}\quad\gamma^{-1}_X(x)\in I_{s,t}\;,\]
and $W_0(x,y)=0$ otherwise (in the displayed expression, $x\mod 1$ stands for $x-\lfloor x\rfloor$).
Finally, we set $W_0(x,y)=0$ for all $(x,y)\in X\times (C_*\cup D_*)$ where $X\in B_*\setminus B_{G_*}$.

\begin{figure}[ht]
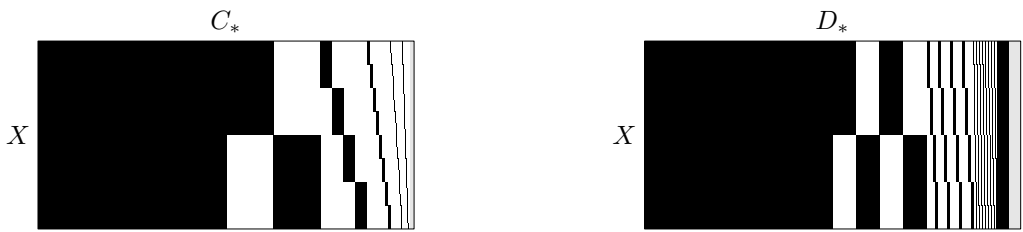

\begin{center}
\epsfbox{epsuniversal-33.mps} \hskip 25mm
\epsfbox{epsuniversal-34.mps}
\end{center}
\caption{The composite tiles $X\times C_*$ (in the left) and $X\times D_*$ (in the right) for $X\in A_*\cup B_{G_*}$.}
\label{referencing_fig1}
\end{figure}

\begin{figure}[ht]
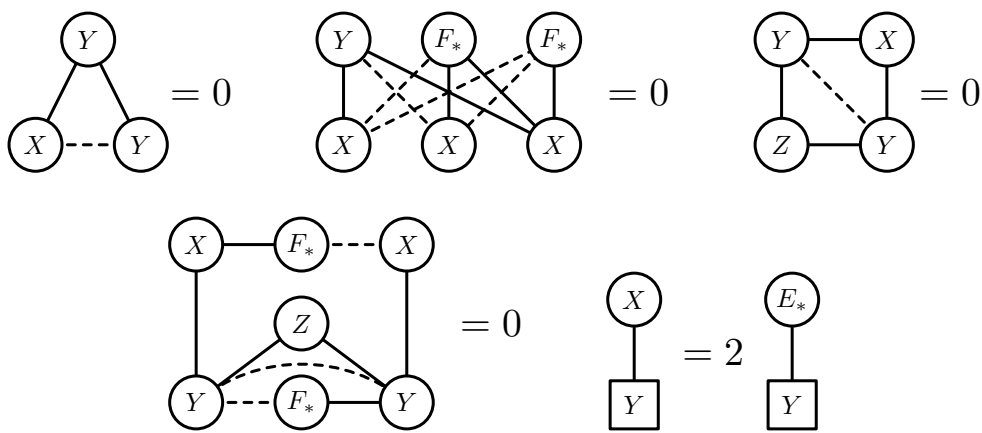

\begin{center}
\epsfbox{epsuniversal-30.mps} \hskip 10mm
\epsfbox{epsuniversal-29.mps} \hskip 10mm
\epsfbox{epsuniversal-31.mps} \vskip 5mm
\epsfbox{epsuniversal-32.mps} \hskip 10mm
\epsfbox{epsuniversal-28.mps} 
\end{center}
\caption{The decorated constraints forcing the structure of the tiles $X\times C_*$ and $X\times D_*$ for $X\in A_*\cup B_{G_*}$.
         The constraints should hold for every such $X$ and for both $(Y,Z)=(C_*,E_*)$ and $(Y,Z)=(D_*,C_*)$.}
\label{referencing_fig2}
\end{figure}

We now establish that the constraints depicted in Figure~\ref{referencing_fig2}
force the structure of the tiles defined in the previous paragraph.
Fix $X\in A_*\cup B_{G_*}$ and fix $(Y,Z)$ to be either $(C_*,E_*)$ or $(D_*,C_*)$.
Recall the definition of $\JJ_{C_*}$ and $\JJ_{D_*}$ from Subsection~\ref{section_indices} and,
for completeness, define $\JJ_{E_*}$ to be $\{I_s,s\in\NN\}$.
Note that each of $\JJ_{C_*}$, $\JJ_{D_*}$ and $\JJ_{E_*}$ is a partition of $[0,1)$ into half-open intervals,
$\JJ_{C_*}$ refines $\JJ_{E_*}$, and $\JJ_{D_*}$ refines $\JJ_{C_*}$.
In particular, the partition given by $\JJ_{Y}$ refines the partition given by $\JJ_{Z}$.
The first two constraints on the first line imply that
there exists a collection of subintervals $K_J$ of $[0,1)$, $J\in\JJ_Y$, such that 
the following holds for almost every $(x,y)\in X\times Y$:
$W(x,y)=1$ if there exists $J\in\JJ_Y$ such that $\gamma_Y^{-1}(g(y))\in J$ and $\gamma_X^{-1}(g(x))\in K_J$ (note that the interval $J$ is uniquely determined by $y$), and
$W(x,y)=0$ otherwise.
Also note that these two constraints do not imply that the intervals $K_J$ are disjoint, and
indeed, we will see further that they are not.
Suppose that $J,J'\in\JJ_Y$, $J\not=J'$, are subintervals of the same interval contained in $\JJ_Z$.
The third constraint on the first line implies that $K_J$ and $K_{J'}$ are disjoint (possibly after removing a set of measure zero from each of them), and
the first constraint on the second line implies that if $J$ precedes $J'$, then $K_J$ precedes $K_{J'}$.
The final constraint yields that
$|K_J|=2^{-s+1}$ for every interval $J\in\JJ_Y$ such that $J\sqsubseteq I_s$, $s\in\NN$.
However, this can only be possible if for every interval $J_0\in\JJ_Z$,
the intervals $K_J$ for $J\sqsubseteq J_0$, $J\in\JJ_Y$, partition the interval $[0,1)$ up to a set of measure zero.
We conclude that $W(x,y)=W_0(g(x),g(y))$ for almost every $(x,y)\in X\times Y$.

\begin{figure}[ht]
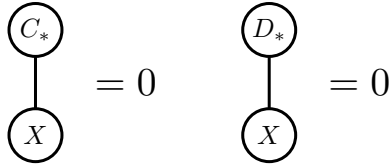

\begin{center}
\epsfbox{epsuniversal-59.mps} \hskip 10mm
\epsfbox{epsuniversal-60.mps} 
\end{center}
\caption{The decorated constraints forcing the structure of the tiles $X\times C_*$ and $X\times D_*$ for $X\in B_*\setminus B_{G_*}$.}
\label{referencing_fig3}
\end{figure}

Finally, the two constraints depicted in Figure~\ref{referencing_fig3} imply that
$W(x,y)=0=W_0(g(x),g(y))$ for almost every $(x,y)\in X\times (C_*\cup D_*)$ where $X\in B_*\setminus B_{G*}$.

\subsection{Forcing the graphon \texorpdfstring{$\boldsymbol{W_F}$}{WF}}
\label{forcing_densities}

In this subsection, we force that the subgraphon of $W$ induced by $A_*$ is weakly isomorphic to $W_F$.
In particular, we define $W_0(\gamma_{A*}(x),\gamma_{A*}(y))=W_F(x,y)$ for every $(x,y)\in [0,1)^2$, and
argue that the constraints presented in this subsection imply that
$W(x,y)=W_0(g(x),g(y))$ for almost every $(x,y)\in A_*\times A_*$.

The argument is split into two steps.
First, we establish that the density of every dyadic square of each tile $A_i\times A_j$ is the same
as the density of the corresponding dyadic square of the tile $B_{G_i}\times B_{G_j}$, $i,j\in [M]$.
This is achieved using the auxiliary structure of the tiles involving $C_*$, $D_*$ and $E_*$.
Recall the partitions $\JJ_{C_*}$, $\JJ_{D_*}$ and $\JJ_{E_*}$ of $[0,1)$ used earlier.
Informally speaking,
we use the partition given by $\JJ_{E_*}$ to determine the order of the dyadic square that we are concerned with,
the partition given by $\JJ_{C_*}$ to determine the row index of the dyadic square and
the partition given by $\JJ_{D_*}$ to determine the column index of the dyadic square.
However, even if the densities of all dyadic squares of tiles $A_i\times A_j$ and $B_{G_i}\times B_{G_j}$ agree,
it does not follow that subgraphons induced by $A_*$ and $B_{G_*}$ are weakly isomorphic.
The partition into dyadic squares could average (hide) a finer internal structure of the tile $A_i\times A_j$.
By Lemma~\ref{lm-oper}, if this were the case,
then the parameter $t(C_4,\cdot)$, which corresponds to the non-induced density of $C_4$,
would be higher for the subgraphon induced by $A_*$ than for the subgraphon induced by $B_{G_*}$,
Hence, fixing the parameter $t(C_4,\cdot)$ together with fixing the densities of all dyadic squares
forces the subgraphons induced by $A_*$ and $B_{G_*}$ to be weakly isomorphic.
We present these arguments formally in the rest of the subsection.

\begin{figure}
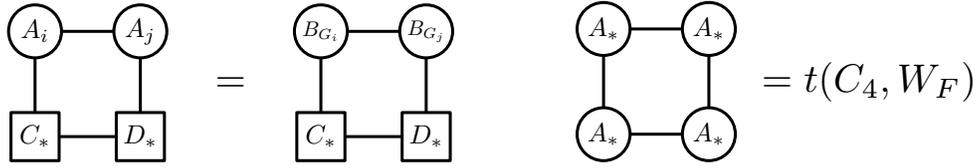

\begin{center}
\epsfbox{epsuniversal-38.mps} \hskip 15mm
\epsfbox{epsuniversal-39.mps} 
\end{center}
\caption{The decorated constraints forcing the structure of the tile $ A_*\times A_*$.
         The first constraint should hold for all $i\in [M]$ and $j\in [M]$.}
\label{densities_fig1}
\end{figure}

Fix $i,j\in [M]$ and consider the first constraint in Figure~\ref{densities_fig1}.
The constraint implies that the following holds for almost every $x\in C_*$ and $y\in D_*$ with $W(x,y)>0$.
Let $s\in\NN$ be the unique integer such that $\gamma_{D_*}^{-1}(g(y))\in I_s$, and
let $t,t'\in [2^{s-1}]$ be the unique integers such that
\[\frac{\gamma_{D_*}^{-1}(g(y))-\min I_s}{|I_s|}\in I_{s,t}\qquad\mbox{and}\qquad
  \left(\frac{\gamma_{D_*}^{-1}(g(y))-\min I_s}{|I_s|}\cdot 2^{s-1}\mod 1\right)\in I_{s,t'}\;.\]
Because of the structure of the tile $C_*\times D_*$ and $W(x,y)>0$,
it follows (with the exceptions forming a set of measure zero) that
\[\gamma_{C_*}^{-1}(g(x))\in I_s\qquad\mbox{and}\qquad \frac{\gamma_{C_*}^{-1}(g(x))-\min I_s}{|I_s|}\in I_{s,t}\;,\]
i.e., the part of $\JJ_{C_*}$ that $\gamma_{C_*}^{-1}(g(x))$ belongs to is determined by $y$.
The indices $s$, $t$ and $t'$,
which are determined by the choice of the $C_*$-root and the $D_*$-root,
describe a dyadic square of $A_i\times A_j$.
For the particular choice of the roots $x\in C_*$ and $y\in D_*$,
the left side of the constraint is equal to
\[ \frac{d_W(g^{-1}(\gamma_{A_i}(I_{s,t})),g^{-1}(\gamma_{A_j}(I_{s,t'})))}{|A_i||A_j|}\;,\]
which corresponds to the density of the dyadic square $I_{s,t}\times I_{s,t'}$ inside the tile $A_i\times A_j$ ($d_W(\cdot,\cdot)$ was defined in Section~\ref{sec-prelim}).
Similarly, the right side is equal to
\[ \frac{d_W(g^{-1}(\gamma_{B_{G_i}}(I_{s,t})),g^{-1}(\gamma_{B_{G_j}}(I_{s,t'})))}{|B_{G_i}||B_{G_j}|}\;,\]
which corresponds to the density of the dyadic square $I_{s,t}\times I_{s,t'}$ inside the tile $B_{G_i}\times B_{G_j}$.
Considering all possible choices of the roots $x\in C_*$ and $y\in D_*$,
we conclude that the following holds for all $s\in\NN$ and $t,t'\in [2^{s-1}]$:
\[ \frac{d_W(g^{-1}(\gamma_{A_i}(I_{s,t})),g^{-1}(\gamma_{A_j}(I_{s,t'})))}{|A_*||A_*|}=
   \int_{\gamma^{-1}_{A_*}(\gamma_{A_i}(I_{s,t}))\times \gamma^{-1}_{A_*}(\gamma_{A_j}(I_{s,t'}))} W_F(x,y)\diff{x}\diff{y},\]
i.e., the density of any dyadic square inside the tile $A_i\times A_j$ is equal to the desired density.
Since $\gamma_{A_*}$ is a linear function and
any half-open subinterval of $[0,1)$ can be expressed as a countable union of half-open intervals of the form $\gamma^{-1}_{A_*}(\gamma_{A_i}(I_{s,t}))$, $i\in [M]$, $s\in\NN$ and $t\in [2^{s-1}]$,
we obtain that the following holds 
\begin{equation}
 \frac{d_W(g^{-1}(\gamma_{A_*}(J)),g^{-1}(\gamma_{A_*}(J')))}{|A_*||A_*|}=
 \int_{J\times J'} W_F(x,y)\diff{x}\diff{y}\label{eq-maineq1}
\end{equation}
for any two two measurable subsets $J$ and $J'$ of $[0,1)$.

Next, fix any measurable bijection $\psi:[0,1)\to A_*$ such that $\psi^{-1}(X)=|X|/|A_*|$ for every measurable subset of $A_*$.
Define a graphon $W_A$ as $W_A(x,y)=W(\psi(x),\psi(y))$,
i.e., $W_A$ is the subgraphon of $W$ induced by $A_*$.
Further, let ${\wg}:[0,1)\to [0,1)$ be the map defined as ${\wg}(x)=\gamma_{A_*}^{-1}(g(\psi(x)))$.
Note that ${\wg}$ is a measure-preserving map from $[0,1)$ to $[0,1)$.
Using \eqref{eq-maineq1}, we obtain that
\begin{equation}
d_{W_A}\left({\wg}^{-1}(J),{\wg}^{-1}(J')\right)=d_{W_F}\left(J,J'\right)\label{eq-maineq}
\end{equation}
for any two measurable subsets $J$ and $J'$ of $[0,1)$.
In addition, the second constraint in Figure~\ref{densities_fig1} yields that
\begin{equation}
t(C_4,W_A)=t(C_4,W_F).\label{eq-maineqt}
\end{equation}
Lemma~\ref{lm-oper} now implies that $W_A(x,y)=W_F({\wg}(x),{\wg}(y))$ for almost every $(x,y)\in [0,1)^2$.
Since it holds that $W_0(\gamma_{A_*}(x),\gamma_{A_*}(y))=W_F(x,y)$ for every $(x,y)\in [0,1)^2$,
we conclude that
\[W(x,y)=W_F(\gamma_{A_*}^{-1}(g(x)),\gamma_{A_*}^{-1}(g(y)))=W_0(g(x),g(y))\]
for almost every $(x,y)\in A_*\times A_*$.

\subsection{Degree balancing}
\label{section_balancing}

We now define the graphon $W_0$ on the remaining tiles except for those involving the part $G_2$.
First, set $W_0(x,y)=0$ for all $(x,y)\in A_*\times B_*$.
The graphon $W_0$ is now defined on all tiles except for those involving the part $G_1$ or $G_2$.
Recall the definition of pre-degrees given in Table~\ref{table2}.
For $x\in [0,1)\setminus (G_1\cup G_2)$, we define
\[h(x)=\int_{[0,1)\setminus (G_1\cup G_2)}W_0(x,y)\diff y,\]
and set 
\[W_0(x,y)=\frac{2}{\eps}\left(\pdeg(X)-h(x)\right)\]
for every $(x,y)\in X^0\times G_1^0$, $X\in\PP\setminus\{G_1,G_2\}$;
we will show that $W_0(x,y)\in [0,1]$ in what follows.
Finally, let $\rho\in [0,1]$ be such that
\[\frac{\rho\eps}{2}+\sum_{X\in\PP\setminus\{G_1,G_2\}}\frac{2}{\eps}\int_X\left(\pdeg(X)-h(x)\right)\diff x\]
is a rational number, and
set $W_0(x,y)=\rho$ for every $(x,y)\in G_1^0\times G_1^0$.
Note that the sum in the displayed expression corresponds to the density between $G_1$ and
all other parts except for $G_1$ and $G_2$.

We now establish that all the values of $W_0$ defined in this subsection belong to $[0,1]$
by showing that
$\pdeg(X)-h(x)\in [0,\eps/2]$ for every $x\in X^0$, $X\in\PP\setminus\{G_1,G_2\}$.
Since the total measure of the parts $B_*$, $C_*$, $D_*$, $E_*$ and $F_*$ is $\eps/4$,
it is enough to show that
\begin{equation}
\int_{A^0}W_0(x,y)\diff y\in\left[\pdeg(X)-\eps/2,\pdeg(X)-\eps/4\right] \label{eq-int}
\end{equation}
for every $x\in X^0$, $X\in\PP\setminus\{G_1,G_2\}$.

If $X=A^0_k$, $k\in [M]$, the value of the integral in \eqref{eq-int}
belongs to $\left[(1-\eps)\frac{k-1}{M},(1-\eps)\frac{k}{M}\right]$ for every $x\in X^0$
by the definition of $Q_k$, which was given at the beginning of this section.
Since $M=4\cdot\frac{1-\eps}{\eps}$, it follows that
the value of the integral belongs to the interval $\left[\frac{(k-1)\eps}{4},\frac{k\eps}{4}\right]$,
which coincides with the interval on the right side of \eqref{eq-int}.
Similarly, if $X=F^0_k$,
the value of the integral in \eqref{eq-int}
belongs to $\left[(1-\eps)\frac{k-1}{M},(1-\eps)\frac{k}{M}\right]$ for every $x\in X^0$
by the definition of $W_0$, and
this interval again coincides with that on the right side of \eqref{eq-int}.

If $X\in B_*$, then the integral in \eqref{eq-int} is zero and \eqref{eq-int} is also satisfied.
If $X=C^0_k$ or $X=D^0_k$, $k\in [m]$,
then the integral in \eqref{eq-int} is equal to $\frac{1-\eps}{2^{k-1}}$ for every $x\in X^0$ and
\eqref{eq-int} is satisfied.
If $X=E^0_k$, $k\in [m-1]$,
then the integral in \eqref{eq-int} is equal to $\frac{1-\eps}{2^{k}}$ for every $x\in X^0$ and
\eqref{eq-int} is again satisfied.
Finally, if $X\in\{C^0_\infty,D^0_\infty,E^0_\infty\}$,
then the integral in \eqref{eq-int} is at most $\frac{1-\eps}{2^{m}}$, and
so its value belongs to the interval on the right side of \eqref{eq-int}.

\begin{figure}[htb]
\begin{center}
\epsfbox{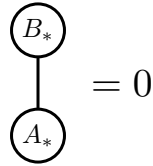}
\end{center}
\caption{The decorated constraint forcing the structure of the tile $A_*\times B_*$.}
\label{cleaning_fig1}
\end{figure}

\begin{figure}[htb]
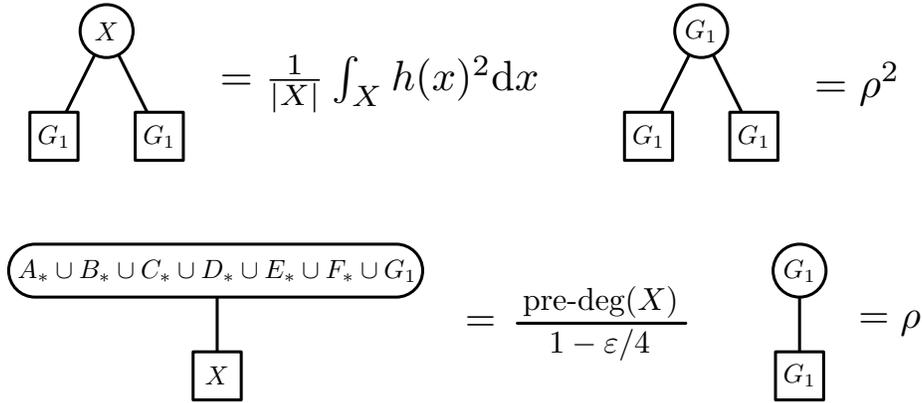

\begin{center}
\epsfbox{epsuniversal-42.mps} \hskip 10mm
\epsfbox{epsuniversal-62.mps} \vskip 10mm
\epsfbox{epsuniversal-41.mps} \hskip 10mm
\epsfbox{epsuniversal-61.mps}
\end{center}
\caption{The decorated constraints forcing the structure of the tiles involving the part $G_1$.
         The first constraint should hold for every $X\in\PP\setminus\{G_1,G_2\}$ and
	 the last constraint for every $X\in\PP\setminus\{G_2\}$.}
\label{balancing_fig1}
\end{figure}

We now force the structure of the tiles that we have just defined.
First, the constraint in Figure~\ref{cleaning_fig1} implies that $W(x,y)=W_0(g(x),g(y))=0$ for almost every $(x,y)\in A\times B$.
We now analyze the constraints depicted in Figure~\ref{balancing_fig1}.
The two constraints on the first line imply that, for every $X\in\PP\setminus\{G_2\}$, the integral
\begin{equation}
\int_{X}W(x,y)W(x,y')\diff x
\label{eq-prod}
\end{equation}
is the same for almost every $y,y'\in G_1$ (and is equal to the value of the corresponding integral in $W_0$).
By Lemma~\ref{lemma_int}, the integral
\[\int_{X}W(x,y)^2\diff x\]
is the same for almost every $y\in G_1$ and its value is equal to \eqref{eq-prod}.
Therefore, for almost every $(y,y') \in G_1^2$, we have that
\[\int_X \left(W(x,y)-W(x,y')\right)^2\diff{x}=0
.\]
In particular, there exists a $y' \in G_1$ such that for almost every $y \in G_1$, 
\[\int_X \left(W(x,y)-W(x,y')\right)^2\diff{x}=0
.\]
This is equivalent to saying that $W(x,y)=W(x,y')$ for almost all $(x,y)\in X\times G_1$.
Thus,
there exists a function ${\widetilde h}:[0,1)\setminus G_2\to\RR$ such that
$W(x,y)={\widetilde h}(x)$ for almost every $(x,y)\in  ([0,1)\setminus G_2)\times G_1$.
The two constraints on the second line in the figure imply that 
\[\int_{G_1}W(x,y)\diff y=\int_{G_1^0}W_0(x,g(y))\diff y\]
for almost every $x\in [0,1)\setminus G_2$.
We conclude that $W(x,y)=W_0(g(x),g(y))$ for almost every $(x,y)\in ([0,1)\setminus G_2)\times G_1$.

\subsection{Degree distinguishing}
\label{section_distinguishing}

We now finish the definition and forcing of the graphon $W_0$.
Recall that we fixed irrational numbers $\delta_X\in (0,\eps/4)$ for each part $X\in\PP$ such that
the numbers $\delta_X$, $X\in\PP$, are rationally independent.
For each $X\in\PP$, we set $W_0(x,y)=\delta_X/|G_2|=4\delta_X/\varepsilon$ for all $(x,y)\in X^0\times G^0_2$.
Observe that the degree of each part $X\in\PP\setminus\{G_1,G_2\}$ is equal to $\pdeg(X)+\delta_X$,
the degree of $G_1$ is $r+\delta_{G_1}$ where $r$ is a rational number (this follows from the choice of $\rho$), and
the degree of $G_2$ is a rational combination of the values $\delta_X$, $X\in\PP$ (recall that $\eps$ is rational). 
Since the values $\delta_X$, $X\in\PP$, are rationally independent,
the degrees of all the parts are distinct.

\begin{figure}[htbp]
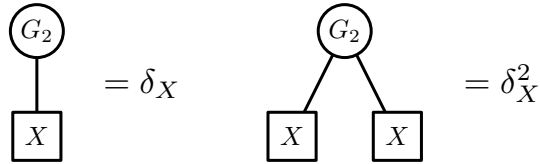

\begin{center}
\epsfbox{epsuniversal-43.mps} \hskip 10mm
\epsfbox{epsuniversal-44.mps} 
\end{center}
\caption{The decorated constraints forcing the structure of the tiles involving the part $G_2$.
         The constraints should hold for every $X\in\PP$.}
\label{distinguish_fig1}
\end{figure}

Fix $X\in\PP$ and consider the constraints depicted in Figure~\ref{distinguish_fig1}.
The first constraint yields that
\[\frac{1}{|G_2|}\int_{G_2}W(x,y)\diff y=\delta_X\]
for almost every $x\in X$, and
the second constraint yields that
\[\frac{1}{|G_2|}\int_{G_2}W(x,y)W(x',y)\diff y=\delta_X^2\]
for almost every $x,x'\in X$.
The latter implies by Lemma~\ref{lemma_int} that
\[\frac{1}{|G_2|}\int_{G_2}W(x,y)^2\diff y=\delta_X^2\]
for almost every $x\in X$.
Hence, for almost every $x\in X$, 
\[\frac{1}{|G_2|}\int_{G_2}\left(W(x,y)-\delta_X\right)^2\diff y=\frac{1}{|G_2|}\int_{G_2}W(x,y)^2\diff y-\frac{2\delta_X}{|G_2|}\int_{G_2}W(x,y)\diff y + \delta_X^2=0
,\]
which implies that $W(x,y)=\delta_X=W_0(g(x),g(y))$ for almost every $(x,y)\in X\times G_2$.
This concludes the argument that $W(x,y)=W_0(g(x),g(y))$ for almost every $(x,y)\in [0,1)^2$ and
the proof of Theorem~\ref{thm1} is now finished.

\section{Sizes of Forcing Families}
\label{sect-unifamily}

In this section, we show that there is no finite family $\GG$ of graphs such that
Theorem~\ref{thm1} would hold for all graphons $W_F$ and $\eps>0$ with $\GG$ being the forcing family.
In particular, the main result of this section is the following theorem.

\begin{thm}
\label{noUniversal}
For every positive integer $n$,
there exists a graphon $W_F$ and a real number $\eps>0$ such that
if $W$ is a finitely forcible graphon containing $W_F$ as a subgraphon on at least $1-\varepsilon$ fraction of its vertices,
then every forcing family for $W$ contains a graph of order greater than $n$. 
\end{thm}

To prove Theorem~\ref{noUniversal},
we need a modification of a result of Erd\H{o}s, Lov\'{a}sz and Spencer~\cite[Lemma 5]{ELS}.
The proof follows the lines of the proof in~\cite{ELS} but we include its sketch for completeness.

\begin{lem}
\label{ELSlem}
Let $n$ be a positive integer and let $H_1,\ldots,H_m$ be all connected graphs on at most $n$ vertices.
There exist graphons $W_1,\dots,W_m$ such that the vectors
\[(d(H_1,W_i),\ldots,d(H_m,W_i)),\;i\in [m],\]
are linearly independent in $\RR^m$, and
there is no index $i\in [m]$ and non-null set $A\subseteq [0,1]$ such that
$W_i$ is positive almost everywhere on $A\times A$.
\end{lem}

\begin{proof}
Fix an integer $n$ and the graphs $H_1,\ldots,H_m$.
Let $k_i$ be the number of vertices of $H_i$, $i\in [m]$.
For $i\in [m]$ and $\vec{s}_i\in [0,1]^{k_i}$ such that $\vec{s}_{i,1}+\cdots+\vec{s}_{i,k_i}\le 1$,
define $W_{i,\vec{s}_i}$ to be the following step graphon with $k_i+1$ parts $S_1,\ldots,S_{k_i+1}$.
The measure of the part $S_j$, $j\in [k_i]$, is $\vec{s}_{i,j}$, and
the measure of the remaining part $S_{k_i+1}$ is $1-(\vec{s}_{i,1}+\cdots+\vec{s}_{i,k_i})$.
The value of $W_{i,\vec{s}_i}(x,y)$ is equal to $1$ for $(x,y)\in [0,1]^2$ such that
$x\in S_j$, $y\in S_{j'}$ and the $j$-th and $j'$-th vertices of $H_i$ are adjacent, and
it is equal to $0$ elsewhere.
Note that there does not exist a non-null set $A\subseteq [0,1]$ such 
that $W_{i,\vec{s}_i}$ is positive almost everywhere on $A\times A$.

Let $\mathcal S \subseteq \RR^m$ be the set of vectors that arise as $(d(H_j,W_{i,\vec{s}_i}))_{j \in [m]}$ over all choices of $i$ and $\vec{s}_i$. We claim that the span of $\mathcal S$ 
is $\RR^m$. Suppose not. Then we may choose $(c_1,\dots,c_m)$ to be a non-zero vector in the orthogonal complement of the span of $\mathcal{S}$.
That is, $c_1,\dots,c_m$ are real numbers, not all zero, such that
\begin{equation}
\label{dependent}
c_1\cdot d(H_1,W_{i,\vec{s}_i})+\cdots +c_m\cdot d(H_m,W_{i,\vec{s}_i})=0
\end{equation}
for every $i\in [m]$ and $\vec{s}_i\in [0,1]^{k_i}$ such that $\vec{s}_{i,1}+\cdots+\vec{s}_{i,k_i}\le 1$.
Take $i$ such that $c_i \ne 0$.
The left side of \eqref{dependent} is a polynomial in $\vec{s}_{i,1},\ldots,\vec{s}_{i,k_i}$.
Observe that the only term that contributes to the coefficient of the monomial $\vec{s}_{i,1}\cdots \vec{s}_{i,k_i}$
is the term $c_i d(H_i,W_{i,\vec{s}_i})$.
It follows that the left side of \eqref{dependent} is polynomial that is not identically zero,
which implies that the equality \eqref{dependent} cannot hold for all choices of $\vec{s}_i\in [0,1]^{k_i}$. 
Thus, the span of $\mathcal S$ is $\RR^m$, and so we can choose $m$ linearly independent vectors from $\mathcal S$. 
The corresponding graphons can be taken as $W_i$.
\end{proof}

We are now ready to prove Theorem~\ref{noUniversal}.

\begin{proof}[Proof of Theorem~\ref{noUniversal}]
Fix an integer $n$.
Note that Proposition~\ref{densall} yields that
the densities of all subgraphs on at most $n$ vertices are determined
by the densities of connected subgraphs on at most $n$ vertices.
Therefore, to prove the theorem, it is enough to show the following:
there exists a graphon $W_F$ and a real number $\eps>0$ such that
no graphon $W$ that contains $W_F$ as a subgraphon on at least $1-\eps$ fraction of its vertices
is a finitely forcible graphon such that the set of all \emph{connected} graphs with at most $n$ vertices is a forcing family. 
Let $H_1,\ldots,H_m$ be all connected graphs with the number of vertices between two and $n$, and
let $k_i$ be the number of vertices of $H_i$, $i\in [m]$.

Let $W_1,\dots,W_m$ be the graphons from Lemma~\ref{ELSlem}.
In addition, let $W_{m+1}$ be the graphon equal to one everywhere.
We define $W_F(x,y)$ to be equal to
\[W_i((m+2)x-(i-1),(m+2)y-(i-1)) \qquad\mbox{if $(x,y)\in \left[\frac{i-1}{m+2},\frac{i}{m+2}\right)^2$ for $i\in [m+1]$, and}\]
equal to $0$ otherwise.
In other words, $W_F$ contains each of the graphons $W_1,\ldots,W_{m+1}$ on a $\frac{1}{m+2}$ fraction of its vertices, and
it is zero elsewhere.

We next define a $m\times m$ square matrix $A^{\delta}$ for $\delta\in (0,1)$ as
\[A^{\delta}_{ij}=(1-\delta)^{k_i}k_i\left(\frac{1}{m+2}\right)^{\;k_i-1}d(H_i,W_j)\mbox{ for } i,j\in [m].\]  
The matrix $A^{\delta}$ is invertible
since multiplying each row by $(1-\delta)^{-k_i}k_i^{-1}\left(\frac{1}{m+2}\right)^{-(k_i-1)}$
results in the matrix that has $d(H_i,W_j)$ as the entry in the $i$-th row and $j$-th column.
Hence, there exists $\eps_0$ such that any matrix obtained from $A^{\delta}$ for $\delta\in (0,\varepsilon_0)$
by perturbing each of its entries by at most $\eps_0$ is invertible.
We now set $\eps=\min\left\{\frac{\eps_0}{(m+3)^{n}n},\frac{1}{m+4}\right\}$.

Suppose that $W$ is a graphon that
contains $W_F$ on a $1-\eps'$ fraction of its vertices for some $\varepsilon'\leq \varepsilon$.
By applying a suitable measure preserving transformation,
we can assume that the subgraphon of $W$ on $\left[\frac{i-1}{m+2}(1-\eps'),\frac{i}{m+2}(1-\eps')\right)^2$
is weakly isomorphic to $W_i$ for every $i\in [m+1]$, and
the graphon $W$ is zero almost everywhere else on $[0,1-\eps')^2$.
Consider a vector $\vec{s}\in\left[0,\frac{1}{m+1}\right)^{m+1}$, and
let $t_0,\ldots,t_{m+2}\in [0,1]$ be such that $t_0=0$, $t_i=t_{i-1}+\vec{s}_i$ for $i\in [m+1]$, and $t_{m+2}=1$.
Define a function $\varphi_{\vec s}:[0,1]\to [0,1]$ as follows:
\[\varphi_{\vec s}(x)=\begin{cases}
  \frac{(i-1)(1-\eps')}{m+2}+\frac{x-t_{i-1}(1-\eps')}{(m+2)(t_i-t_{i-1})} & \mbox{if $x\in [t_{i-1}(1-\eps'),t_i(1-\eps'))$ for $i\in [m+2]$, and} \\
  x & \mbox{otherwise.}
  \end{cases}\]
Finally, define the graphon $W_{\vec{s}}$ as $W_{\vec{s}}(x,y)=W(\varphi_{\vec s}(x),\varphi_{\vec s}(y))$.
Informally speaking, the part of $W$ containing $W_i$ is stretched to size $\vec{s}_i(1-\eps')$ for every $i\in [m+1]$.
In particular, the graphons $W$ and $W_{\frac{1}{m+2},\ldots,\frac{1}{m+2}}$ are the same.

We now analyze $d(H_i,W_{\vec{s}})$ as a function of $\vec{s}_1,\ldots,\vec{s}_{m+1}$.
Each of the $k_i$ vertices of $H_i$ can be chosen 
either from one of the $m+2$ intervals $[t_{i-1}(1-\eps'),t_i(1-\eps'))$, $i\in [m+2]$, or
from the interval $[1-\eps',1]$,
i.e., there are $(m+3)^{k_i}$ ways how individual vertices of $H_i$ can be chosen from these $m+3$ intervals.
Since $H_i$ is connected,
the choices where no vertex of $H_i$ is chosen to be in the interval $[1-\eps',1]$ contribute to $d(H_i,W_{\vec{s}})$ by a total of
\[\sum_{j=1}^{m+1}((1-\eps')\vec{s}_j)^{\;k_i}d(H_i,W_j).\]
We next analyze contributions of the choices where at least one of the vertices of $H_i$ is chosen from the interval $[1-\eps',1]$.
Consider a choice where $\ell\ge 1$ vertices are mapped to the interval $[1-\eps',1]$.
For this choice, the contribution is a product of a constant between $0$ and $1$,
the $\ell$-th power of $\eps'$ and
$k_i-\ell$ factors of the form $\vec{s}_1,\vec{s}_2,\ldots,\vec{s}_{m+1}$, and $(1-\vec{s}_1-\vec{s}_2-\ldots-\vec{s}_{m+1})$;
the constant is the probability that a $W$-random graph is $H_i$ conditioned on the vertices being sampled from the chosen intervals.
Since the absolute value of the derivative of the product of the $k_i-\ell$ factors with respect to $\vec{s}_j$ is at most $k_i-\ell$ and
the product of the remaining factors is a constant, which is at most $\eps'$,
the contribution to the derivative of $d(H_i,W)$ with respect to $\vec{s}_j$ is at most $k_i\eps'$.
This holds for each such choice of intervals for the vertices of $H_i$.
It follows that the derivative of $d(H_i,W)$ with respect to $\vec{s}_j$ is between
\[(1-\eps')^{k_i}k_i\vec{s}_j^{\;k_i-1}d(H_i,W_j)-(m+3)^{k_i}k_i\eps'\mbox{ and }
  (1-\eps')^{k_i}k_i\vec{s}_j^{\;k_i-1}d(H_i,W_j)+(m+3)^{k_i}k_i\eps'\mbox{.}\]

Observe that
there is an open ball contained in $\left(0,\frac{1}{m+1}\right)^{m+1}$ that
contains the point $\vec{s}=\left(\frac{1}{m+2},\ldots,\frac{1}{m+2}\right)$ such that
each $d(H_i,W_{\vec{s}})$, $i\in [m]$, is well-defined on this ball.
Also observe that
the entries of the Jacobian matrix $\left(\frac{\partial d(H_i,W_{\vec{s}})}{\partial\vec{s}_j}\right)_{i,j\in [m]}$
for $\vec{s}=\left(\frac{1}{m+2},\ldots,\frac{1}{m+2}\right)$
differ from the entries of $A^{\varepsilon'}$ by at most $(m+3)^{k_i}k_i\eps' \le (m+3)^{n}n\eps \le \eps_0$.
In particular, the Jacobian matrix is invertible.
Hence, the Implicit Function Theorem implies that
there exist $\delta\in\left(0,\frac{1}{m+1}-\frac{1}{m+2}\right)$ and
a continuous function $g\colon \left(\frac{1}{m+2}-\delta, \frac{1}{m+2}+\delta\right)\to \left(0,\frac{1}{m+1}\right)^m$ such that
$g\left(\frac{1}{m+2}\right)=\left(\frac{1}{m+2},\ldots,\frac{1}{m+2}\right)$ and
\[d(H_i,W)=d(H_i,W_{(g(z)_1,\ldots,g(z)_m,z)})\]
for every $z\in\left(\frac{1}{m+2}-\delta,\frac{1}{m+2}+\delta\right)$.
Fix $z\in\left(\frac{1}{m+2},\frac{1}{m+2}+\delta\right) \subseteq \left(0,\frac{1}{m+1}\right)$. 
We set $W'=W_{(g(z)_1,\ldots,g(z)_m,z)}$.
Observe that the densities of all graphs $H_1,\ldots,H_m$ are the same in $W$ and $W'$
by the choice of $g$ and $z$.

We finish the proof by establishing that the graphons $W$ and $W'$ are not weakly isomorphic.
By Lemma~\ref{lm-omega}, it suffices to show that $\omega(W)<\omega(W')$.
Note that $\omega(W)\ge\frac{1-\varepsilon'}{m+2}$,
which can be seen by considering the interval $[\frac{m}{m+2}(1-\varepsilon'),\frac{m+1}{m+2}(1-\varepsilon'))$.
Let $X$ be any measurable set $X\subseteq [0,1]$ such that $W$ is equal to $1$ almost everywhere on $X\times X$ and
$|X|\ge\frac{1-\varepsilon'}{m+2}$.
We construct a measurable set $X'\subseteq [0,1]$ such that $|X'|>|X|$ and $W'$ is equal to $1$ almost everywhere on $X'\times X'$.

Define $X_1=X\cap [0,1-\eps')$ and $X_2=X\cap [1-\eps',1]$.
Observe that $X_1\sqsubseteq\left[\frac{m}{m+2}(1-\varepsilon'),\frac{m+1}{m+2}(1-\varepsilon')\right)$
by the second assertion of Lemma~\ref{ELSlem} and
that $|X_2|\le\varepsilon'\le\eps\le\frac{1}{m+4}$.
This implies that 
\[|X_1|\ge \frac{1-\eps'}{m+2}-\varepsilon' =\frac{1-(m+3)\varepsilon'}{(m+2)} \ge \frac{1}{(m+2)(m+4)}.\]
Since the interval $\left[\frac{m}{m+2}(1-\varepsilon'),\frac{m+1}{m+2}(1-\varepsilon')\right)$ is stretched to an interval of size $z(1-\varepsilon')>\frac{1}{m+2}(1-\varepsilon')$ in $W'$,
we obtain that
\[\omega(W')\ge z(m+2)|X_1|+|X_2|=|X|+(z(m+2)-1)|X_1|\ge |X|+\frac{z(m+2)-1}{(m+2)(m+4)}.\]
Since the choice of $X$ was arbitrary, it follows that
\[\omega(W')\ge\omega(W)+\frac{z(m+2)-1}{(m+2)(m+4)} > \omega(W)\]
and we conclude that the graphons $W$ and $W'$ are not weakly isomorphic.
\end{proof}

We would like to remark that Theorem~\ref{noUniversal} excludes the existence of a finite family $\GG$ of graphs
such that for every graphon $W_F$ and every $\eps>0$, there exists a finitely forcible graphon $W_0$ that
contains $W_F$ as subgraphon on a $1-\eps$ fraction of its vertices, and $\GG$ is a forcing family for $W_0$.
In other words, Theorem~\ref{thm1} cannot be proven with a universal forcing family (unlike Theorem~\ref{thm-ckm}). However,
we were not able to show that the number of graphs needed to force the structure of graphons containing $W_F$ must grow with $\eps^{-1}$,
i.e., we do not know whether the following stronger statement is true:
for every $K\in\NN$, there exist a graphon $W_F$ and $\eps>0$ such that
if $W_0$ is a finitely forcible graphon that contains $W_F$ as a subgraphon on a $1-\eps$ fraction of its vertices,
then every forcing family of $W_0$ contains at least $K$ graphs.

\section*{Acknowledgment}

We would like to thank Jan Hladk\'y for useful discussions on the step forcing property of $C_4$,
in particular, the arguments used at the end of Subsection~\ref{forcing_densities}. We would also like to thank
the authors of~\cite{bib-dolezal+} for sharing an early version of their manuscript.

\newcommand{\advances}{Adv. Math. }
\newcommand{\annals}{Ann. of Math. }
\newcommand{\cpc}{Combin. Probab. Comput. }
\newcommand{\dcg}{Discrete Comput. Geom. }
\newcommand{\discrete}{Discrete Math. }
\newcommand{\eur}{European J.~Combin. }
\newcommand{\gfa}{Geom. Funct. Anal. }
\newcommand{\jcta}{J.~Combin. Theory Ser.~A }
\newcommand{\jctb}{J.~Combin. Theory Ser.~B }
\newcommand{\rsa}{Random Structures Algorithms }
\newcommand{\sidma}{SIAM J.~Discrete Math. }
\newcommand{\tcs}{Theoret. Comput. Sci. }

\bibliographystyle{amsplain}


\begin{dajauthors}
\begin{authorinfo}[dan]
  Daniel Kr\'{a}l'\\
  Faculty of Informatics\\
  Masaryk University\\
  Brno, Czech Republic\\
  dkral\imageat{}fi\imagedot{}muni\imagedot{}cz \\
  \url{http://fi.muni.cz/~dkral}
\end{authorinfo}
\begin{authorinfo}[laci]
  L{\'a}szl{\'o} M. Lov{\'a}sz\\
  Department of Mathematics\\
  Massachusetts Institute of Technology\\
  Cambridge, MA, USA\\
  lmlovasz\imageat{}mit\imagedot{}edu \\
  \url{http://math.mit.edu/~lmlovasz/}
\end{authorinfo}
\begin{authorinfo}[jon]
  Jonathan A. Noel\\
  Mathematics Institute\\
  University of Warwick\\
  Coventry, UK\\
  j\imagedot{}noel\imageat{}warwick\imagedot{}ac\imagedot{}uk\\
  \url{http://homepages.warwick.ac.uk/staff/J.Noel}
\end{authorinfo}
\begin{authorinfo}[kuba]
  Jakub Sosnovec\\
  NXP semiconductors\\
  Brno, Czech Republic\\
  j\imagedot{}sosnovec\imageat{}email\imagedot{}cz\\
  \url{http://sites.google.com/site/jsosnovec1}
\end{authorinfo}
\end{dajauthors}

\end{document}